\documentclass[final]{siamart171218}
\usepackage[top=2.73cm,bottom=2.73cm,left=2.73cm,right=2.73cm]{geometry}
\usepackage{amsfonts,bm}
\usepackage{graphicx}
\usepackage{epstopdf}
\usepackage{algorithmic}
\usepackage{cases}
\usepackage{subfigure,multirow}

\makeatletter
\@addtoreset{Equation}{Section}
\makeatother
\ifpdf
\DeclareGraphicsExtensions{.eps,.pdf,.png,.jpg}
\else
\DeclareGraphicsExtensions{.eps}
\fi
\theoremstyle{remark}

\definecolor{darkcerulean}{rgb}{0.1, 0.1, 0.7}
\usepackage{hyperref}
\hypersetup{
	colorlinks,
	citecolor=darkcerulean,
	filecolor=darkcerulean,
	linkcolor=darkcerulean,
	urlcolor=darkcerulean
}

\usepackage{stmaryrd}
\usepackage{hvdashln} 
	\usepackage{color}
	\usepackage{tikz,pgfplots,pgf}
	\usetikzlibrary{matrix,shapes,arrows,positioning}
	\usepackage{etoolbox} 
	\usetikzlibrary{shapes}
	\tikzset{box/.style ={
			rectangle, 
			rounded corners =5pt, 
			minimum width =50pt, 
			minimum height =20pt, 
			inner sep=5pt, 
			draw=blue} 
}

\tikzset{zbox/.style ={
		rectangle, 
		minimum width =50pt, 
		minimum height =20pt, 
		inner sep=5pt, 
		draw=black} 
}

\tikzset{ball/.style ={
	circle, 
	minimum width =20pt, 
	minimum height =20pt, 
	inner sep=0.1pt, 
	draw=blue} 
}

\tikzset{global scale/.style={
scale=#1,
every node/.append style={scale=#1}
}
}

\newcommand{\TheTitle}{\Large Accelerated primal-dual methods with enlarged step sizes and operator learning for nonsmooth optimal control problems}
\newcommand{\TheAuthors}{{Yongcun Song, Xiaoming Yuan, Hangrui Yue}}

\title{{\TheTitle}\thanks{July 24, 2023
\funding{The work of  Y.S was supported by the Humboldt Research Fellowship for postdoctoral researchers. The work of X. Y was partially supported by the theme-based research scheme T32-707/22-N. The work of H.Y was supported by the Fundamental Research Funds for the Central Universities, Nankai University (Grant Number 63221035). }
}}

\author{Yongcun Song\thanks{Chair for Dynamics, Control, Machine Learning and Numerics$-$Alexander von Humboldt Professorship, Department of Mathematics,  Friedrich-Alexander-Universit\"at Erlangen-N\"urnberg, 91058 Erlangen, Germany, \email{ysong307@gmail.com}}
	\and
	Xiaoming Yuan\thanks{Department of Mathematics, The University of Hong Kong, Pok Fu Lam, Hong Kong, \email{xmyuan@hku.hk}}
	\and
	Hangrui Yue\thanks{School of Mathematical Sciences, Nankai University, Tianjin 300071, China, \email{yuehangrui@gmail.com}}
}

\usepackage{amsopn}

\ifpdf
\hypersetup{
	pdftitle={\TheTitle},
	pdfauthor={\TheAuthors}
}
\fi
\begin{document}
	
\maketitle
\begin{abstract}
We consider a general class of nonsmooth optimal control problems with partial differential equation (PDE) constraints, which are very challenging due to their nonsmooth objective functionals and the resulting high-dimensional and ill-conditioned systems after discretization. We focus on the application of a primal-dual method, with which different types of variables can be treated individually in iterations and thus its main computation at each iteration only requires solving two PDEs. Our target is to accelerate the primal-dual method with either enlarged step sizes or operator learning techniques. The accelerated primal-dual method with enlarged step sizes improves the numerical performance of the original primal-dual method in a simple and universal way, while its convergence can be still proved rigorously. For the operator learning acceleration, we construct deep neural network surrogate models for the involved PDEs. Once a neural operator is learned, solving a PDE requires only a forward pass of the neural network, and the computational cost is thus substantially reduced. The accelerated primal-dual method with operator learning is mesh-free, numerically efficient, and scalable to different types of PDEs. The acceleration effectiveness of these two techniques is promisingly validated by some preliminary numerical results.
\end{abstract}

\section{Introduction}
Optimal control problems with partial differential equation (PDE) constraints  play a crucial role in various areas such as
physics, chemistry, biology, engineering, and finance. We refer the reader to \cite{glowinski1994exact, glowinski1995exact, glowinski2008exact, hinze2008optimization, lions1971optimal, troltzsch2010optimal} for a few references. Typically, additional nonsmooth constraints are imposed on the control (variable) to promote some desired properties such as boundedness, sparsity, and discontinuity; see  \cite{elvetun2016,kunisch2004,lions1971optimal,stadler2009elliptic,troltzsch2010optimal} and references therein. These optimal control problems are numerically challenging. One particular reason is that the PDE constraints and other nonsmooth constraints on the control are coupled together and  the resulting algebraic systems after discretization are usually high-dimensional and ill-conditioned. To solve such a nonsmooth optimal control problem, it is desirable to consider different types of variables separately in iterations so that a nonsmooth optimal control problem can be decoupled into some much easier subproblems while there is no need to solve any computationally demanding system after discretization. For this purpose, it becomes necessary to deliberately consider the structure of the model under discussion for algorithmic design.

In this paper, we study the algorithmic design for optimal control problems that can be uniformly and abstractly formulated as
\begin{equation}\label{Basic_Problem}
		\min_{u\in {U}, y\in {Y}} \quad   \frac{1}{2}\| y-y_d \|^2_Y+\frac{\alpha}{2}\|u\|_U^2+ \theta(u) \\\qquad
		{\hbox{s.t.}}~  y=Su,
\end{equation}
    where $u\in U$ and $y\in Y$, with $U$ and $Y$ being function spaces, are called the control and the state, respectively; $y_d\in Y$ is a given target; and $\alpha>0$ is a regularization parameter. In addition, $y=Su$ represents a linear PDE, in which $S: U\rightarrow Y$ is the corresponding solution operator. The nonsmooth convex functional $\theta(u): U\rightarrow\mathbb{R}$ is employed to impose some additional constraints on the control, such as boundedness \cite{lions1971optimal,troltzsch2010optimal}, sparsity \cite{stadler2009elliptic}, and discontinuity \cite{elvetun2016}. Various optimal control problems with PDE constraints can be covered by (\ref{Basic_Problem}). For instance, the abstract state equation $y=Su$ can be specified as the parabolic equation \cite{glowinski1994exact}, the elliptic equation \cite{hintermuller2002primal}, the wave equation \cite{glowinski1995exact}, etc. Also, the control $u$ can be a distributed control or a boundary control. Moreover, the functional $\theta(u)$ can be the indicator function of an admissible set \cite{glowinski2022}, the $L^1$-regularization term \cite{stadler2009elliptic}, or the total variation regularization \cite{elvetun2016,kunisch2004}.

\subsection{State-of-the-art}
Numerical study for optimal control problems, including (\ref{Basic_Problem}), has become an increasingly active field in the past decades. In the literature, the semi-smooth Newton (SSN) methods have been studied intensively and extensively; see, e.g., \cite{hintermuller2002primal,hinze2008optimization,kroner2011,KR2002,ulbrich2011semismooth} for control constrained optimal control problems, and \cite{stadler2009elliptic} for sparse elliptic optimal control problems. In \cite{hintermuller2002primal}, it has been proved that SSN methods possess locally superlinear convergence and they usually can find high-precision solutions provided that some initial values are appropriately chosen. Computationally, it is notable that, at each iteration of SSN methods, one encounters a large-scale and ill-conditioned Newton system. Practically, it is required to solve these Newton systems up to very high accuracy to ensure the convergence, which is numerically challenging, especially for time-dependent problems \footnote{The same concerns also apply to interior point methods, e.g., \cite{pearson2017}, for different types of
optimal control problems.}; see, e.g., \cite {pearson2012regularization,porcelli2015preconditioning,schiela2014operator,stoll2013one} for more discussions.

On the other hand, the alternating direction method of multipliers (ADMM) \cite{gabay1975dual,glowinski1975approximation} has  been applied to various optimal control problems modeled by (\ref{Basic_Problem}); see, e.g.  \cite{attouch2008augmented,GSY2019,glowinski2022,pougkakiotis2020,zhang2017}.  At each iteration of the ADMM metods, one only needs to solve a simple optimization subproblem, which generally has a closed-form solution, and a standard unconstrained optimal control problem. The dimensionality of the unconstrained optimal control subproblems after discretization is inevitably high. Hence, these subproblems must be solved iteratively while it is computationally expensive to acquire accurate solutions of these subproblems. To tackle this computation bottleneck, an inexact ADMM was proposed in \cite{glowinski2022}, with an automatically adjustable inexactness criterion for the inner iterations. As a result, the unconstrained optimal control subproblems only need to be solved up to a rather low accuracy by a few inner iterations, while the overall convergence of the inexact ADMM can be still guaranteed rigorously. Notwithstanding, we reiterate that optimal control subproblems still have to be solved for the inexact ADMM in \cite{glowinski2022}.

\subsection{Vanilla application of the primal-dual method}
To solve the problem (\ref{Basic_Problem}) efficiently, we aim at such algorithms that can avoid both complicated Newton systems and unconstrained optimal control subproblems. For this purpose, it suffices to apply the primal-dual method in \cite{chambolle2011first}, because it does not require specific initial iterates and its resulting subproblems are easier than the original model. The primal-dual method in \cite{chambolle2011first} and its variants have been widely applied in various areas such as PDEs \cite{liu2023jcp,liu2023arxiv}, imagining processing \cite{chambolle2011first}, statistical learning \cite{goldstein2015adaptive}, and inverse problems \cite{biccari2022,clason2017primal,tian2018convergence,valkonen2014primal}.

We shall show that, the vanilla application of the primal-dual method in \cite{chambolle2011first} to the abstract model (\ref{Basic_Problem}) requires solving two PDEs at each iteration, and thus it differs from the just-mentioned SSN and ADMM approaches in the literature. To fix ideas, we focus on a parabolic control constrained optimal control problem in the following discussion and all the results can be easily extended to other  problems modeled by (\ref{Basic_Problem}).

Let $\Omega\subset\mathbb{R}^d (d \geq 1)$ be a bounded domain and $\Gamma:=\partial \Omega$ its boundary. We consider the following parabolic optimal control problem:
  \begin{equation}\label{optimal_control}
  	\min_{u\in {L^2(\mathcal{O})}, y\in {L^2(Q)}} \frac{1}{2}\| y-y_d \|^2_{L^2(Q)}+\frac{\alpha}{2}\|u\|_{L^2(\mathcal{O})}^2+ \theta(u)
  \end{equation}
  subject to the parabolic problem
  \begin{flalign}\label{state_eqn}
  	\begin{aligned}
  		\frac{\partial y}{\partial t}-\Delta y=u\chi_{\mathcal{O}} ~\text{in}~ \Omega\times(0,T), \quad
  		y=0 ~\text{on}~ \Gamma\times(0,T),\quad
  		y(0)=0.
  	\end{aligned}
  \end{flalign}
  Above, $Q=\Omega\times(0,T)$ with $0<T<+\infty$; $\mathcal{O}=\omega\times(0,T)$ with $\omega$ an open subset of $\Omega$; $\chi_\mathcal{O}$ is the characteristic function of $\mathcal{O}$; $y_d\in L^2(Q)$ is a given target; and $\alpha>0$ is the regularization parameter. Moreover, we specify $\theta(u)$ as the indicator function of the admissible set $U_{ad}$:
  \begin{equation}\label{admissible_set}
  	U_{ad}:=\{v\in L^\infty(\mathcal{O})| a\leq v(x,t)\leq b ~\text{a.e. in} ~\Omega\times(0,T)\},
  \end{equation}
  with $a$ and $b$ given constants.  Existence and uniqueness of the solution to problem (\ref{optimal_control})-(\ref{admissible_set}) have been well studied in \cite{lions1971optimal,troltzsch2010optimal}.

To apply the primal-dual method in \cite{chambolle2011first} to problem (\ref{optimal_control})-(\ref{admissible_set}), we first observe that  (\ref{optimal_control})-(\ref{admissible_set}) can be rewritten as
    \begin{equation}\label{primalproblem}
    \begin{aligned}
    \min_{u\in L^2(\mathcal{O})} f(Su)+g(u),
    \end{aligned}
    \end{equation}
    where $ f(Su)= \frac{1}{2}\left\| Su-y_d \right\|_{L^2(Q)}^2$ and $g(u)=\frac{\alpha}{2}\|u\|_{L^2(\mathcal{O})}^2+\theta(u)$.
    Introducing an auxiliary variable $p\in {L^2(Q)}$, it follows from the Fenchel duality \cite{bauschke2011} that the primal-dual formulation of (\ref{primalproblem}) reads as
    \begin{equation}\label{saddlepoint}
    \min_{u\in L^2(\mathcal{O})}\max_{p\in {L^2(Q)}}  g(u)+(p, Su)_{L^2(Q)}-f^*(p),
    \end{equation}
    where $(\cdot,\cdot)_{L^2(Q)}$ denotes the canonical $L^2$-inner product, $f^*(p):={\sup}_{y\in{L^2(Q)}}\{(y,p)_{L^2(Q)}-f(y)\}$ is the convex conjugate of $f(y)$ and can be specified as
    $f^*(p)=\frac{1}{2}\|p\|_{L^2(Q)}^2+(p, y_d)_{L^2(Q)}$.
    Then, implementing the primal-dual method in \cite{chambolle2011first} to (\ref{saddlepoint}), we readily obtain the following scheme:
    \begin{numcases}
       ~u^{k+1}=\arg\min_{u\in L^2(\mathcal{O})} \{g(u)+({p}^k,S u)_{L^2(Q)}+\frac{1}{2r}\|u-u^k\|_{L^2(\mathcal{O})}^2\},\label{pd1}\\
       \bar{u}^k=2u^{k+1}-u^k,\label{pd2}\\
    p^{k+1}=\arg\max_{p\in L^2(Q)}\{(p, S{\bar{u}}^k)_{L^2(Q)}-f^*(p)-\frac{1}{2s}\|p-p^k\|_{L^2(Q)}^2\}\label{pd3},
    \end{numcases}
    where the parameters $r>0$ and $s>0$ can be understood as the step sizes of the primal and dual subproblems, respectively.

For the solutions to subproblems (\ref{pd1}) and (\ref{pd3}), one can show that
\begin{equation}\label{u_k+1}
	u^{k+1}=P_{U_{ad}}\left(-\frac{S^*{p}^k-\frac{1}{r}u^k}{\alpha+\frac{1}{r}}\right),\quad
	p^{k+1}=\left(S(2u^{k+1}-u^k)+\frac{1}{s}p^k-y_d\right)/(1+\frac{1}{s}),
\end{equation}
where $S^*: L^2(Q)\rightarrow L^2(\mathcal{O})$ is the adjoint operator of $S$, and $P_{U_{ad}}(\cdot)$ denotes the projection onto the admissible set $U_{ad}$, namely, $P_{U_{ad}}(v)(x,t)  := \max\{a, \min\{v(x,t), b\}\}$ a.e in $\mathcal{O}, \forall v\in L^2(\mathcal{O})$.
It follows from (\ref{u_k+1}) that the main computation cost of (\ref{pd1})-(\ref{pd3}) consists of solving $y^{k}:=S(2u^{k+1}-u^k)$, i.e., the state equation (\ref{state_eqn}) with $u=2u^{k+1}-u^k$, and computing $q^k|_{\mathcal{O}}:=S^*{p}^k$, where $q^k$ is obtained by solving the adjoint equation:
\begin{equation}\label{adjointP}
	-\frac{\partial q^k}{\partial t}
	-\Delta q^k={p^k}~ \text{in}
	~\Omega\times(0,T),~
	q^k=0~ \text{on}~ \Gamma\times(0,T),~
	q^k(T)=0.
\end{equation}

Obviously, the main computation of (\ref{pd1})-(\ref{pd3}) is solving only the PDEs (\ref{state_eqn}) and (\ref{adjointP}). The Newton systems of SSN methods and the unconstrained optimal control subproblems of ADMM methods are both completely avoided. Therefore, the computational load of the primal-dual method (\ref{pd1})-(\ref{pd3}) at each iteration is much lower than that of the SSN and ADMM methods.  Meanwhile, when $\theta(u)$ is an $L^1$ or a total variation regularization, methods for solving the resulting subproblem (\ref{pd1}) can be found in Section \ref{sec:numerical} and \cite{song2023admmpinns}. It can be seen that, for these two cases, the primal-dual method (\ref{pd1})-(\ref{pd3}) also only requires solving two PDEs as shown in (\ref{u_k+1}).

\subsection{Enlarging step sizes for (\ref{pd1})--(\ref{pd3})}
As analyzed in \cite{chambolle2011first}, to ensure the convergence of the primal-dual method (\ref{pd1})--(\ref{pd3}),  the step sizes $r$ and $s$ are required to satisfy the condition
\begin{equation}\label{convergence_condition}
	r\cdot s<\frac{1}{\|S\|^2},
\end{equation}
where
$
\|S\|={\sup}_{\|v\|_{L^2(\mathcal{O})}=1}\{\|Sv\|_{L^2(Q)}, \forall v\in L^2(\mathcal{O})\}
$. Numerical efficiency of the primal-dual method (\ref{pd1})--(\ref{pd3}) certainly depends on the choices of the step sizes $r$ and
$s$. In the literature, to allow for larger step sizes and hence accelerate convergence, $r\cdot s$ are usually chosen to be very close to, or even equal the upper bound $\frac{1}{\|S\|^2}$. It is thus clearly interesting to discuss if the upper bound $\frac{1}{\|S\|^2}$ can be further enlarged theoretically, while the convergence of the primal-dual method (\ref{pd1})--(\ref{pd3}) can be still guaranteed.
Recently, it has been shown in \cite{he2022} that, for saddle point problems in the generic convex setting, the convergence condition (\ref{convergence_condition}) can be optimally improved to
$$
r\cdot s < \frac{4}{3} \cdot \frac{1}{\|S \|^2}.
$$
We are motivated by the work \cite{he2022} and consider whether or not the upper bound $\frac{4}{3} \cdot \frac{1}{\|S \|^2}$ can be further enlarged for the model (\ref{Basic_Problem}), given that the functionals $ \frac{1}{2}\left\| Su-y_d \right\|_{L^2(Q)}^2$ and $\frac{\alpha}{2}\|u\|_{L^2(\mathcal{O})}^2$ are indeed strongly convex. Below, we shall show that, to ensure the convergence of the primal-dual method (\ref{pd1})--(\ref{pd3}) for problem (\ref{Basic_Problem}), the step sizes $r$ and $s$ can be chosen subject to
\begin{equation}\label{o_con_bound}
	r\cdot s < \frac{4 + 2 \alpha r}{3} \cdot \frac{1}{\|S \|^2}.
\end{equation}
As a result, the step sizes $r$ and $s$ can be enlarged for the primal-dual method (\ref{pd1})--(\ref{pd3}) and its numerical performance can be accelerated. With (\ref{o_con_bound}), the primal-dual method (\ref{pd1})--(\ref{pd3}) is accelerated simply by a larger interval for possible choices of the step sizes, and the computational load of each iteration remains. As we shall show, this is a simple and universal way for accelerating the primal-dual method (\ref{pd1})--(\ref{pd3}) by reducing its number of iterations, while its convergence can be still proved rigorously.

\subsection{Accelerating (\ref{pd1})--(\ref{pd3}) with operator learning}
In the context of traditional numerical methods, the PDEs (\ref{state_eqn}) and (\ref{adjointP}) should be solved repeatedly by certain mesh-based numerical discretization schemes (e.g., finite difference methods (FDM) or finite element methods (FEM)), which require solving large-scale and ill-conditioned algebraic systems. Even a single implementation of such a PDE solver could be expensive; hence the computation cost for solving the PDEs (\ref{state_eqn}) and (\ref{adjointP}) repeatedly is usually extremely high. Furthermore, given another target $y_d\in L^2(Q)$, one has to solve the resulting optimal control problem from scratch, and hence solve the state and adjoint equations repeatedly again.

To tackle the computational difficulty above, we advocate to adopt deep learning techniques, which have recently emerged as a new powerful tool for scientific computing problems thanks to the universal approximation property and great expressibility of deep neural networks (DNNs). Indeed, various deep learning approaches have been developed for PDEs; see e.g. \cite{beck2019,e2018,han2018,lu2021learning,ludeepxde2021,raissi2019physics,sirignano2018dgm,wang2021} and references therein. Compared with traditional numerical solvers for PDEs, these deep learning techniques are usually mesh-free, easy to implement, and very flexible to different PDEs. It is arguably accepted that deep learning techniques are helping alleviate human efforts in algorithmic design yet empowering the solvability of a large class of scientific computing problems. Among deep learning techniques, one is to approximate PDE solutions via DNNs, such as the deep Ritz method \cite{e2018}, the deep Galerkin method \cite{sirignano2018dgm},  and physic-informed neural networks \cite{ludeepxde2021,pang2019,raissi2019physics,yu2022gPINN}. Despite that these methods have shown promising results in diversing applications, each of them is tailored for a specific PDE. It is thus necessary to train a new neural network given a different input function (e.g., initial condition, boundary condition, or source term), which is computationally costly and time-consuming. Hence, these methods are not applicable to (\ref{state_eqn}) and (\ref{adjointP}) because they have to be solved repeatedly with different $u^k$ and $p^k$.

Another deep learning technique, called operator learning, is to apply a DNN to approximate the solution operator of a PDE, which maps from an input function to the PDE solution, see e.g., \cite{kocachki2021,li2020FNO,lu2021learning,wang2021}. To be concrete, consider a PDE solution operator  $G  : X \rightarrow  Y, v \mapsto w,$ where $X$ and $Y$ are two infinite-dimensional Banach spaces and $w = G (v)$. Operator learning aims at approximating $G$ with a neural network $\mathcal{G}_\theta$ parameterized by $\theta$. Once a neural solution operator is learned, obtaining a PDE solution $\mathcal{G}_\theta (v)$ for a new input function $v$ requires only a forward pass of the neural network. Hence, neural solution operators can be used as effective surrogates for PDEs and are computationally attractive for problems that require repetitive yet expensive simulations, see e.g., \cite{hwang2021solving,lye2021iterative,wang2021fast}.

We are thus inspired to consider constructing two DNN surrogates for the PDEs (\ref{state_eqn}) and (\ref{adjointP}) by operator learning to accelerate the primal-dual method (\ref{pd1})-(\ref{pd3}). Precisely, we propose to construct two neural surrogates $y=\mathcal{S}_{\theta_s}(u)$ and $q=\mathcal{S}_{\theta_a}(p)$ parameterized by $\theta_s$ and $\theta_a$ for (\ref{state_eqn}) and (\ref{adjointP}), respectively. Then, replacing $S$ and $S^*$ by $\mathcal{S}_{\theta_s}$ and $\mathcal{S}_{\theta_a}$ in (\ref{u_k+1}), we propose the following primal-dual method with operator learning for solving (\ref{optimal_control})-(\ref{admissible_set}):
\begin{equation}\label{pdol}
u^{k+1}=P_{U_{ad}}\left(-\frac{\mathcal{S}_{\theta_a}({p}^k)-\frac{1}{r}u^k}{\alpha+\frac{1}{r}}\right), \quad p^{k+1}=\left(\mathcal{S}_{\theta_s}(2u^{k+1}-u^k)+\frac{1}{s}p^k-y_d\right)/(1+\frac{1}{s}).
\end{equation}

Different primal-dual methods can be specified from (\ref{pdol}) by using different operator learning techniques such as the Deep Operator Networks (DeepONets) \cite{lu2021learning}, the physic-informed DeepONets \cite{wang2021}, the Fourier Neural Operator (FNO) \cite{li2020FNO}, the Graph Neural Operator (GNO) \cite{li2020GNO}, and the Laplace Neural Operator (LNO) \cite{cao2023}. Note that, given two neural surrogates, these primal-dual methods with operator learning only require implementing two forward passes of the neural networks and some simple algebraic operations. More importantly, given a different $y_d\in L^2(Q)$, these primal-dual methods with operator learning can be applied directly to the resulting optimal control problems without the need of solving any PDE. Moreover, we reiterate that the resulting primal-dual methods with operator learning for solving (\ref{optimal_control})-(\ref{admissible_set}) can be easily extended to other various optimal control problems in form of (\ref{Basic_Problem}), see Section \ref{se: pdol} for more details.

Finally, we mention that some deep learning techniques have been recently developed for solving optimal control problems with PDE constraints, such as the ISMO \cite{lye2021iterative}, operator learning methods \cite{hwang2021solving,wang2021fast}, the amortized finite element analysis \cite{xue2020}, and physics-informed neural networks (PINNs) methods \cite{barry2022,haoBilevel2022,mowlayi2021,raissi2019physics,sun2022}. All these deep learning methods, however, are designed for only smooth optimal control problems with PDE constraints, and they cannot be directly applied to the nonsmooth problems modeled by (\ref{Basic_Problem}). To tackle this issue, the ADMM-PINNs algorithmic framework has been recently proposed in \cite{song2023admmpinns}. With the advantages of both the ADMM and PINNs, the ADMM-PINNs algorithmic framework in \cite{song2023admmpinns} is applicable to a wide range of nonsmooth optimal control and inverse problems. It is notable that the neural networks in the ADMM-PINNs algorithmic framework have to be re-trained at each iteration.


\subsection{Organization}
The rest of this paper is organized as follows. In Section \ref{sec:convergence_primal_dual}, we prove the convergence of the primal-dual method (\ref{pd1})-(\ref{pd3}) with the enlarged step sizes (\ref{o_con_bound}). In Section \ref{sec:numerical}, we test a parabolic control constrained optimal control problem and validate the efficiency of the accelerated primal-dual method (\ref{pd1})-(\ref{pd3}) with (\ref{o_con_bound}). In Section \ref{se:se}, we showcase extensions to other optimal control problems by a sparse elliptic optimal control problem. In Section \ref{se: pdol}, we focus on the implementation of the primal-dual method with operator learning (\ref{pdol}), and report some numerical results to validate its efficiency. Finally, some conclusions and perspectives are given in Section \ref{sec:conclusion}.

\section{Convergence analysis of (\ref{pd1})-(\ref{pd3}) with (\ref{o_con_bound})}\label{sec:convergence_primal_dual}

In this section, we rigorously prove the convergence for the primal-dual method (\ref{pd1})--(\ref{pd3}) with the enlarged step sizes (\ref{o_con_bound}). For this purpose, we first show that the primal-dual method (\ref{pd1})--(\ref{pd3}) can be equivalently interpreted as a linearized ADMM. We reiterate that the convergence analysis does not depend on the specific form of the solution operator $S$ and the nonsmooth convex functional $\theta(u)$ in (\ref{Basic_Problem}). Hence, the convergence results can be applied to other optimal control problems with PDE constraints in the form of (\ref{Basic_Problem}).

\subsection{Preliminary}
In this subsection, we summarize some known results in the literature for the convenience of
further analysis. We denote by $(\cdot,\cdot)$ and $\|\cdot\|$ the canonical $L^2$-inner product and the associated norm, respectively.

Let $\lambda\in L^2(Q)$ be the Lagrange multiplier associated with the constraint $y=Su$. It is clear that  problem (\ref{optimal_control})-(\ref{admissible_set}) is equivalent to the following saddle point problem:
\begin{equation}\label{saddle_admm}
\min_{u\in L^2(\mathcal{O}), y \in L^2(Q)} \max_{\lambda \in L^2(Q)} g(u) + f(y) + (\lambda, Su - y).
\end{equation}
Let $(u^*, y^*, \lambda^*)^\top$ be the solution of (\ref{saddle_admm}). Then, the first-order optimality condition of (\ref{saddle_admm}) reads as
\begin{equation}
\left\{\begin{aligned}
\label{opt}
&  \partial \theta(u^*)+\alpha u^* + S^*\lambda^*\ni 0  \, , \\
& y^*-y_d-\lambda^*=0,\\
& -Su^*+ y^*=0,
\end{aligned}\right.
\end{equation}
which can be rewritten as the following variational inequalities (VIs):
\begin{equation}\left\{\begin{aligned}
\label{VIopt}
& \theta(u)-\theta(u^*)+\left(u - u^*,  \alpha u^* + S^*\lambda^*  \right) \geq 0, \, \forall u\in L^2(\mathcal{O}), \\
& \left(y - y^*,  y^*-y_d- \lambda^* \right) \geq 0, \, \forall y \in L^2(Q),\\
& \left(q - \lambda^*, -Su^* + y^* \right) \geq 0, \forall q \in L^2(Q).
\end{aligned}\right.\end{equation}

The following lemma will be used later. Its proof can be found in \cite{bauschke2011}, and thus omitted.

\begin{lemma}[Moreau's identity]
	Let $H$ be a Hilbert space and $\phi: H \rightarrow R \cup \{+\infty\}$ a proper, convex and lower semi-continuous extended real-valued functional on $H$. Let $\phi^*(v):={\sup}_{w\in H}(v,w)-\phi(w)$ be the convex conjugate of $\phi(v)$. Then, for all $w \in H$, it holds that
	\begin{equation}
	\label{moreau1}
	w = \arg\min_v\{\phi(v)+ \frac{1}{2}\| v - w \|^2\} + \arg\min_v\{\phi^*(v)+ \frac{1}{2}\| v - w \|^2\}.
	\end{equation}
	
\end{lemma}

For any constant $s>0$, applying (\ref{moreau1}) to $s\phi(v)$, instead of $\phi(v)$, we have
\begin{equation}
\label{Moreau}
w = \arg\min_v\{\phi(v)+ \frac{1}{2s}\| v - w \|^2\} + s\arg\min_v\{\phi^*(v)+ \frac{s}{2}\| v - \frac{1}{s}w \|^2\}.
\end{equation}

\subsection{Equivalence between (\ref{pd1})--(\ref{pd3})  and linearized ADMM}
In this subsection, we show that the  primal-dual method (\ref{pd1})--(\ref{pd3}) is equivalent to the following linearized ADMM:
\begin{equation}\label{alg}
\left\{\begin{aligned}
& y^{k+1}  = \arg \min_{y \in L^2(Q)}\left\{f(y)-\left(y,  sSu^{k} \right)+ \frac{s}{2}\left\|y- \frac{1}{s}\lambda^k\right\|^{2}\right\}, \\
& u^{k+1}  =\arg \min_{u\in L^2(\mathcal{O})}\left\{g(u)+\left(\lambda^k + s(Su^{k} - y^{k+1}) , S u\right)+\frac{1}{2 r}\left\|u-u^{k}\right\|^{2}\right\},\\
& \lambda^{k+1} = \lambda^k + s(Su^{k+1} - y^{k+1}) .
\end{aligned}\right.\end{equation}

First, we note that the primal-dual method (\ref{pd1})-(\ref{pd3}) can be rewritten as
\begin{numcases}
~u^{k+1}=\arg\min_{u\in L^2(\mathcal{O})} \{g(u)+({p}^k,S u)+\frac{1}{2r}\|u-u^k\|^2\},\label{pds1}\\
p^{k+1}=\arg\max_{p\in L^2(Q)}\{-f^*(p)-\frac{1}{2s}\|p-(p^k+ s(2Su^{k+1}-Su^k))\|^2\}\label{pds2}.
\end{numcases}
Let
$\lambda^{k+1}= p^k + sS(u^{k+1}- u^k)$. Then,  (\ref{pds1})-(\ref{pds2}) can be written as
\begin{numcases}
~u^{k+1}=\arg\min_{u\in L^2(\mathcal{O})} \{g(u)+({p}^k,S u)+\frac{1}{2r}\|u-u^k\|^2\},\label{pdn1}\\
\lambda^{k+1}= p^k + sS(u^{k+1}- u^k),\label{pdn2}\\
p^{k+1}=\arg\max_{p\in L^2(Q)}\{-f^*(p)-\frac{1}{2s}\|p-(\lambda^{k+1}+sSu^{k+1})\|^2\}\label{pdn3}.
\end{numcases}
Next, taking $w=\lambda^{k+1} + sSu^{k+1}$ and $\phi=f^*(p)$ in  (\ref{Moreau}), we obtain that
\begin{equation}
\begin{aligned}
\label{moreau2}
\lambda^{k+1} + sSu^{k+1} =  & \arg \min_{p \in L^2(Q)}\left\{f^{*}(p)-\left(p,  Su^{k+1} \right)+\frac{1}{2 s}\left\|p- \lambda^{k+1}\right\|^{2}\right\}  \\
& + s \arg\min_{y \in L^2(Q)} \left\{f(y)- \left(y, s Su^{k+1} \right) +  \frac{s}{2}\| y - \frac{1}{s} \lambda^{k+1}  \|^2 \right\}.
\end{aligned}
\end{equation}
Clearly, the first term of the right-hand side is exactly $p^{k+1}$ obtained by (\ref{pdn3}). Additionally, we introduce
$$y^{k+2}  = \arg \min_{y \in L^2(Q)}\left\{f(y)-\left(y,  sSu^{k+1} \right)+ \frac{s}{2}\left\|y- \frac{1}{s}\lambda^{k+1}\right\|^{2}\right\}. $$
Then, (\ref{moreau2}) can be rewritten as
$
 \lambda^{k+1} + sSu^{k+1} = p^{k+1} +  sy^{k+2},
$
which implies that $ p^k =  \lambda^k + s(Su^{k} - y^{k+1})$. Substituting  this result into (\ref{pdn2}) and (\ref{pdn3}) to elmiminate $p^{k}$ and $p^{k+1}$, we thus have
\begin{numcases}
~ u^{k+1}  =\arg \min_{u\in L^2(\mathcal{O})}\left\{g(u)+\left(\lambda^k + s(Su^{k} - y^{k+1}) , S u\right)+\frac{1}{2 r}\left\|u-u^{k}\right\|^{2}\right\},\label{b1}\\
 \lambda^{k+1} = \lambda^k + s(Su^{k+1} - y^{k+1}),\label{b2} \\
 y^{k+2}  = \arg \min_{y \in L^2(Q)}\left\{f(y)-\left(y,  sSu^{k+1} \right)+ \frac{s}{2}\left\|y- \frac{1}{s}\lambda^{k+1}\right\|^{2}\right\}\label{b3} .
\end{numcases}
Swapping the order such that the update of $y$ comes first, we get the linearized ADMM (\ref{alg}) directly.

\subsection{Convergence}
In this subsection, we prove the convergence of the primal-dual method (\ref{pd1})-(\ref{pd3}) with the enlarged step sizes (\ref{o_con_bound}) in form of the linearized ADMM (\ref{alg}).

We first see that, for any $y \in L^2(Q)$ and $u\in L^2(\mathcal{O})$, the iterate $(y^{k+1}, u^{k+1}, \lambda^{k+1})^\top$ generated by the linearized ADMM (\ref{alg}) satisfies the following VIs:
{
	\begin{eqnarray}
		\left(y - y^{k+1}, y^k-y_d-s(Su^k -  y^{k+1})- \lambda^k \right) \geq 0, \label{VIorg1} \\
		\theta(u)-\theta(u^{k+1})+(u -  u^{k+1},  \alpha u^{k+1} + S^*\left(\lambda^k + s(Su^{k} - y^{k+1})\right)+ \frac{1}{r}( u^{k+1} - u^k) ) \geq 0,  \label{VIorg2} \\
		\frac{1}{s}(\lambda^{k+1} - \lambda^k) - (Su^{k+1} - y^{k+1}) = 0. \label{VIorg3}
	\end{eqnarray}
}
Though (\ref{alg}) is no more related to $p^k$, for the convenience of further analysis, we still denote
\begin{equation}
	\label{defpk}
	p^{k}= \lambda^k + s(Su^{k} - y^{k+1}).
\end{equation}
Substituting it into (\ref{VIorg1})--(\ref{VIorg3}) yields
\begin{equation}
	\label{VIkplus1}
	\left \{ \begin{aligned}
		& \left(y - y^{k+1},y^{k+1}-y_d -p^{k} \right) \geq 0, \forall y \in L^2(Q), \\
		& \theta(u)-\theta(u^{k+1})+\left(u -  u^{k+1}, \alpha u^{k+1} + S^*p^k + \frac{1}{r}( u^{k+1} - u^k) \right) \geq 0, \forall u\in L^2(\mathcal{O}), \\
		& \left(\lambda - p^{k},  \frac{1}{s}(\lambda^{k+1} - \lambda^k) - (Su^{k+1} - y^{k+1})\right) \geq 0, \forall \lambda \in L^2(Q).
	\end{aligned}
	\right.
\end{equation}
Next, we present some useful lemmas.

\begin{lemma}
	Let $\{(u^k, y^k, \lambda^k)^\top\}$ be the sequence generated by the linearized ADMM (\ref{alg}) and $(u^*, y^*, \lambda^*)^\top$ the solution of (\ref{saddle_admm}). We have
	\begin{equation}
	{\begin{aligned}
	\label{seqU}
	(S (u^{k+1}- u^k), \lambda^{k+1}-\lambda^k ) \geq \frac{1}{r}( u^{k+1} - u^*, u^{k +1} - u^k) + \frac{1}{s}( \lambda^{k+1} - \lambda^*,\lambda^{k+1} - \lambda^k) + \alpha \|u^{k+1}-u^*\|^2 .
	\end{aligned}}
	\end{equation}
\end{lemma}

\begin{proof}
	
	First, taking $(y,u,\lambda)^\top=(y^{k+1}, u^{k+1}, p^{k})^\top$  in (\ref{VIopt}) and $(y,u,\lambda)^\top =  (y^*, u^*, \lambda^*)^\top$ in (\ref{VIkplus1}), respectively, and adding them together, we get
	\begin{equation*}
	\left \{ \begin{aligned}
	& \left(y^* -  y^{k+1}, -p^{k} + \lambda^* \right) \geq 0,\\
	& \left(u^* - u^{k+1}, \alpha u^{k+1} - \alpha u^* + S^* ( p^{k} - \lambda^*) + \frac{1}{r} (u^{k+1} - u^k)\right) \geq 0, \\
	& \left(\lambda^* -  p^{k},   \frac{1}{s}(\lambda^{k+1} - \lambda^k) - S(u^{k+1}-u^*) + (y^{k+1} - y^*) \right) \geq 0.
	\end{aligned}
	\right.
	\end{equation*}
Adding the above three inequalities together, we have
	\begin{equation}\label{s2}
	\frac{1}{r}(u^* -  u^{k+1},  u^{k+1} - u^k) + \frac{1}{s} (\lambda^* - p^k,  \lambda^{k+1}  - \lambda^k) \geq  \alpha \|u^{k+1} - u^* \|^2.
	\end{equation}
From (\ref{b2}) and  (\ref{defpk}), we have $p^{k} = \lambda^{k+1} - s S(u^{k+1}-u^k) $. Then, the desired result (\ref{seqU}) follows from (\ref{s2}) directly.
\end{proof}

\begin{lemma}
Let $\{(u^k, y^k, \lambda^k)^\top\}$ be the sequence generated by the linearized ADMM (\ref{alg}), and $(u^*, y^*, \lambda^*)^\top$ the solution point of (\ref{saddle_admm}). Then, we have
	\begin{equation}
	\label{main_eq}
	\begin{aligned}
	&2(S(u^{k+1}- u^k),\lambda^{k+1}-\lambda^k  ) \\ \geq& [(\frac{1}{r} + \alpha ) \|u^{k+1}- u^*\|^2  + \frac{1}{s}\|\lambda^{k+1}- \lambda^*\|^2]  - [(\frac{1}{r} + \alpha )\|u^{k}- u^*\|^2 + \frac{1}{s}\|\lambda^k- \lambda^*\|^2] \\
	& + [(\frac{1}{r} + \frac{\alpha}{2} )\|u^{k+1}- u^{k}\|^2) + \frac{1}{s}\|\lambda^{k+1}- \lambda^k\|^2)]. \\
	\end{aligned}
	\end{equation}
\end{lemma}

\begin{proof}
	We first note that
	\begin{equation}\label{s3}
	2(u^{k+1} - u^*,  u^{k+1} - u^{k})  = (\|u^{k+1}- u^*\|^2 - \|u^{k}- u^*\|^2) + \|u^{k+1} - u^k \|^2,
	\end{equation}
	and
	\begin{equation}\label{s4}
	2(\lambda^{k+1} - \lambda^*,  \lambda^{k+1} - \lambda^k)   = (\|\lambda^{k+1}- \lambda^*\|^2 - \|\lambda^k- \lambda^*\|^2) + \|\lambda^{k+1} - \lambda^k \|^2.
	\end{equation}
	By the Cauchy-Schwarz inequality, we have
	\begin{equation}\label{s5}
	 \begin{aligned}
	2\|u^{k+1} - u^*\|^2 =& (\|u^{k+1}- u^*\|^2 - \|u^{k}- u^*\|^2) + (\|u^{k+1}- u^*\|^2 + \|u^{k}- u^*\|^2) \\
	& \geq (\|u^{k+1}- u^*\|^2 - \|u^{k}- u^*\|^2) + \frac{1}{2}\|u^{k+1}- u^{k}\|^2.
	\end{aligned} \end{equation}
Substituting the inequalities (\ref{s3}), (\ref{s4}) and (\ref{s5}) into (\ref{seqU}), we obtain the desired result (\ref{main_eq}) directly.
\end{proof}

For convenience, we introduce the notations
\begin{equation}\label{def_d_sig}
D = \frac{1}{r} I - \frac{3 s}{4 + 2 \alpha r} S^*S,~\text{and}~	\sigma = \frac{2 + 4\alpha r}{2 + \alpha r} s.
\end{equation}
It is easy to verify that $D$ is self-adjoint and positive definite under the condition (\ref{o_con_bound}). Moreover, it holds that
\begin{equation}\label{s1}
\frac{1}{r}I - s S^* S =  D - \frac{1}{4} \sigma S^* S.
\end{equation}
With the above result, we can obtain the following estimate.
\begin{lemma}
	Let $\{(u^k, y^k, \lambda^k)^\top\}$ be the sequence generated by the linearized ADMM (\ref{alg}). We have that
	\begin{equation}
	\label{pTAx1}
		{\begin{aligned}
-  (S(u^{k+1}- u^{k}),&\lambda^{k+1}  - \lambda^k)
	\geq   \left[\frac{1}{2} \|u^{k+1} - u^{k}\|^2_D + \frac{1}{8}  \sigma \|S(u^{k+1} - u^{k})\|^2 \right] +  \alpha \|u^{k+1} - u^k \|^2  \\
	& - \left[ \frac{1}{2} \|u^{k} - u^{k-1}\|^2_D + \frac{1}{8} \sigma \|S(u^{k} - u^{k-1})\|^2 \right] - \frac{1}{2} \sigma  \|S(u^{k+1} - u^{k})\|^2.
	\end{aligned}}
	\end{equation}
\end{lemma}

\begin{proof}
	Substituting (\ref{VIorg3}) into (\ref{VIorg2}) to eliminate $y^{k+1}$ , we have
	{
	\begin{equation}
	\label{xkplus1}
	\theta(u)-\theta(u^{k+1})+\left (u - u^{k+1}, \alpha u^{k+1} + S^* \lambda^{k+1} + (\frac{1}{r}I  - s S^*S )(u^{k+1} - u^k) \right) \geq 0, \forall u\in L^2(\mathcal{O}).
	\end{equation}
}
    We relabel the superscript $k+1$ as $k$ in the above VI and obtain
	\begin{equation}
	\label{xk}
	\theta(u)-\theta(u^{k})+\left (u - u^{k}, \alpha u^{k} + S^* \lambda^k + (\frac{1}{r}I  - s S^*S )(u^{k} - u^{k-1}) \right) \geq 0, \forall u\in L^2(\mathcal{O}).
	\end{equation}
	Taking $u = u^k$ in (\ref{xkplus1}) and $u = u^{k+1}$ in (\ref{xk}), and adding the resulting two inequalities, we obtain
	\begin{equation}\label{s6}
	\begin{aligned}
	&-  \left( S(u^{k+1}- u^{k}),\lambda^{k+1}  - \lambda^k\right) \\
	 \geq & \left(u^{k+1} - u^{k},  (\frac{1}{r}I  - s S^* S )\left[(u^{k+1} - u^{k}) - (u^k - u^{k-1})\right] \right) +  \alpha \|u^{k+1} - u^k \|^2 \\
	 \overset{(\ref{s1})}{=}&  \left(u^{k+1} - u^{k},  (D -\frac{1}{4}\sigma S^* S )\left[(u^{k+1} - u^{k}) - (u^k - u^{k-1})\right] \right) +  \alpha \|u^{k+1} - u^k \|^2 \\
	 =&  \|u^{k+1} - u^{k}\|^2_D -  (u^{k+1} - u^{k}, D (u^k - u^{k-1})) - \frac{1}{4}  \sigma  \|S(u^{k+1} - u^{k})\|^2\\
	& \qquad + \frac{1}{4}  \sigma \left(S(u^{k+1} - u^{k}),  S (u^k - u^{k-1})\right) +  \alpha \|u^{k+1} - u^k \|^2 .
	\end{aligned}
	\end{equation}
	Applying the Cauchy-Schwarz inequality to (\ref{s6}), we have
	\begin{equation*}
	\begin{aligned}
	-  \left( S (u^{k+1}- u^{k}),\lambda^{k+1}  - \lambda^k \right)
	\geq ~ & \frac{1}{2} \|u^{k+1} - u^{k}\|^2_D - \frac{1}{2} \|u^{k} - u^{k-1}\|^2_D \\
	&\hspace{-2cm} - \frac{3}{8} \sigma  \|S(u^{k+1} - u^{k})\|^2 - \frac{1}{8}  \sigma \|S(u^{k} - u^{k-1})\|^2  +  \alpha \|u^{k+1} - u^k \|^2,
	\end{aligned}
	\end{equation*}
	which is just the desired result (\ref{pTAx1}).
\end{proof}


\begin{lemma}
	Let $\{(u^k, y^k, \lambda^k)^\top\}$ be the sequence generated by the linearized ADMM (\ref{alg}). Then, for any $\delta \in (0, \frac{1}{2s})$, we have
	\begin{equation}
	\label{pTAx2}
	- \left( S(u^{k+1}- u^{k}),\lambda^{k+1}  - \lambda^k\right) \geq
	- \frac{1}{4} s(1 + 2s\delta) \|S(u^{k+1}- u^k)\|^2 - \frac{1}{s} (1-s\delta) \|\lambda^k - \lambda^{k+1} \|^2.
	\end{equation}
\end{lemma}
\begin{proof}
	By the Cauchy-Schwarz inequality, we have
	\begin{equation*}
	- \left( S(u^{k+1}- u^{k}),\lambda^{k+1}  - \lambda^k\right) \geq
	- \frac{s}{4 (1 -s\delta)} \|S(u^{k+1}- u^k)\|^2 - \frac{1}{s} (1-s\delta) \|\lambda^k - \lambda^{k+1} \|^2.
	\end{equation*}
	For $\delta \in (0, \frac{1}{2s})$, we have $0<\frac{1}{1- s\delta} \leq 1 + 2s\delta$ and thus complete the proof.
\end{proof}

By adding (\ref{main_eq}), (\ref{pTAx1}) and (\ref{pTAx2}), we obtain that
\begin{equation}
\label{left0}
\begin{aligned}
0  \geq& \left[(\frac{1}{r} + \alpha ) \|u^{k+1}- u^*\|^2  + \frac{1}{s}\|\lambda^{k+1}- \lambda^*\|^2 + \frac{1}{2} \|u^{k+1} - u^{k}\|^2_D + \frac{1}{8}  \sigma \|S(u^{k+1} - u^{k})\|^2  \right]    \\
& - \left[(\frac{1}{r} + \alpha )\|u^{k}- u^*\|^2 + \frac{1}{s}\|\lambda^k- \lambda^*\|^2 + \frac{1}{2} \|u^{k} - u^{k-1}\|^2_D + \frac{1}{8} \sigma \|S(u^{k} - u^{k-1})\|^2  \right] \\
& + [(\frac{1}{r} + \frac{3\alpha}{2} )\|u^{k+1}- u^{k}\|^2)   - \frac{1}{4} [2\sigma + s(1 + 2s\delta)] \|S(u^{k+1}- u^k)\|^2  + \delta \|\lambda^k - \lambda^{k+1} \|^2.
\end{aligned}
\end{equation}
To simplify the notations, we introduce
\begin{equation}\label{def_E}
E_k = (\frac{1}{r} + \alpha )\|u^{k}- u^*\|^2 + \frac{1}{s}\|\lambda^k- \lambda^*\|^2 + \frac{1}{2} \|u^{k} - u^{k-1}\|^2_D + \frac{1}{8} \sigma \|S(u^{k} - u^{k-1})\|^2,
\end{equation}
and
\begin{equation*}
V_{k+1} = \delta (\|u^{k+1} - u^k \|^2 + \|\lambda^{k+1} - \lambda^k \|^2).
\end{equation*}
Next, we intend to show that
$
E_{k+1} \leq E_{k} - V_{k+1},
$
which implies $\sum_{k=1}^{\infty} V_k < + \infty$.
Then, the convergence of (\ref{alg}) can be proved from this result.

\begin{lemma}
Let $\{(u^k, y^k, \lambda^k)^\top\}$ be the sequence generated by the linearized ADMM (\ref{alg}) and $(u^*, y^*, \lambda^*)^\top$ the solution point of (\ref{saddle_admm}). We have
	\begin{equation}\label{MainRes}
	E_{k+1}\leq E_k-V_{k+1}.
	\end{equation}
\end{lemma}

\begin{proof}
It follows from the positive definiteness of $D$ that there exists a sufficiently small constant $\delta \in (0, \frac{1}{2s})$ such that
	$ D > \delta (\frac{2}{3\alpha r} I + \frac{s^2}{3 \alpha r}S^*S),$
	which implies that
	\begin{equation}
	\label{assumpD}
	\frac{3\alpha r}{2} \|u^{k+1} - u^k \|^2_D -\frac{s^2\delta}{2} \|S(u^{k+1}- u^k)\|^2 \geq \delta \|u^{k+1}- u^k\|^2.
	\end{equation}
		Recall (\ref{def_d_sig}). We thus have
	\begin{equation}
	\label{DefDx}
	\frac{1}{r} \|u^{k+1} - u^k\|^2 - \frac{3s}{4 + 2 \alpha r} \| S (u^{k+1} - u^k)\|^2 \geq 0.
	\end{equation}
	Moreover, by some simple manipulations, we can show that
	\begin{equation}
	\label{Dcoincidence}
	\frac{3\alpha }{2} \|u^{k+1} - u^k\|^2 + ( \frac{3s}{4 + 2 \alpha r} - \frac{1}{2} \sigma - \frac{1}{4} s) \|S(u^{k+1}- u^k)\|^2
	= \frac{3\alpha r}{2} \|u^{k+1} - u^k \|^2_D.
	\end{equation}
Combining the results (\ref{left0}), (\ref{assumpD}),  (\ref{DefDx}) and (\ref{Dcoincidence}) together, we get (\ref{MainRes}) directly .
\end{proof}

With the help of preceding lemmas, we now prove the strong global convergence of the linearized ADMM (\ref{alg}) under the condition (\ref{o_con_bound}).

\begin{theorem}
	Let $\{(u^k, y^k, \lambda^k)^\top\}$ be the sequence generated by the linearized ADMM (\ref{alg}), and $(u^*, y^*, \lambda^*)^\top$ the solution point of (\ref{saddle_admm}). If $r$ and $s$ satisfy the condition (\ref{o_con_bound}), then $\{u^k\}$ converges to $u^*$ strongly in $L^2(\mathcal{O})$, $\{y^k\}$ converges to $y^*$ strongly in $L^2(Q)$, and $\lambda^k$ converges to $\lambda^*$ strongly in $L^2(Q)$.
\end{theorem}

\begin{proof}
	Summarizing the inequality (\ref{MainRes}) from $k = 1$ to $k = \infty$, we have that
	\begin{equation*}
		{\small
	\begin{aligned}
	\delta \sum_{k= 1}^{\infty} (\|u^{k+1} - u^k \|^2 + \|\lambda^{k+1} - \lambda^k \|^2)  \leq  (\frac{1}{r} + \alpha )\|u^{1}- u^*\|^2  + \frac{1}{s}\|\lambda^{1}- \lambda^*\|^2 + \frac{1}{2} \|u^{1} - u^{0}\|^2_D + \frac{1}{8}  \sigma  \|S(u^{1} - u^{0})\|^2
	<  + \infty.
	\end{aligned}}
	\end{equation*}
	As a result, we have $\|u^{k+1} - u^k\|\rightarrow 0$ and $\|\lambda^{k+1} - \lambda^k\| \rightarrow 0$, and $\{u^k\}$ and $\{\lambda^k\}$ are bounded in $L^2(\mathcal{O})$ and $L^2(Q)$, respectively. Recall (\ref{seqU}). We have
		$$
		\begin{aligned}
		(S (u^{k+1}- u^k), \lambda^{k+1}-\lambda^k ) \geq& \frac{1}{r}( u^{k+1} - u^*, u^{k +1} - u^k) + \frac{1}{s}( \lambda^{k+1} - \lambda^*,\lambda^{k+1} - \lambda^k) + \alpha \| u^* -  u^{k+1} \|^2,
		\end{aligned}
		$$
	    which implies that
		\begin{equation*}
		\begin{aligned}
		\alpha \| u^* -  u^{k+1} \|^2 \leq  \|S\| \|u^{k+1}- u^k\|\| \lambda^{k+1}-\lambda^k \| &+ \frac{1}{r} \|u^{k+1} - u^*\|\| u^{k +1} - u^k \|+ \frac{1}{s}\| \lambda^{k+1} - \lambda^*\|\|\lambda^{k+1} - \lambda^k\|    .
	    \end{aligned}
		\end{equation*}
		Since the solution operator $S$ and the iterates $\lambda^k$ and $u^k$ are bounded, it follows from $\|u^{k+1} - u^k\|\rightarrow 0$ and $\|\lambda^{k+1} - \lambda^k\| \rightarrow 0$ that
		$u^{k} \rightarrow u^*$ strongly in $L^2(\mathcal{O}).$
		
		It follows from the continuity of the operator $S$ that $Su^{k} \rightarrow Su^*$ strongly in $L^2(\mathcal{O})$.
		Additionally, the fact $\|\lambda^{k+1} - \lambda^k\| \rightarrow 0$  implies that $\|Su^{k+1}-y^{k+1}\|\rightarrow 0$, and hence
		$y^k\rightarrow y^*$ strongly in  $L^2(Q).$
		
		Concerning with the convergence of $\lambda^k$, we note that $\lambda^*=y^*-y_d$ (see (\ref{opt})) and it follows from the optimality condition of the $y$-subproblem in (\ref{alg}) that
		$
		\lambda^k=-sSu^k+y^{k+1}-y_d+sy^{k+1}.
		$
		We thus have that
		$$
		\|\lambda^k-\lambda^*\|=\|-s(Su^k-y^k)+(y^{k+1}-y^*)\|\leq s\|Su^k-y^k\|+\|y^{k+1}-y^*\|.
		$$
		Since $\|Su^{k}-y^{k}\|\rightarrow 0$ and $\|y^{k+1}-y^*\|\rightarrow 0$, we conclude that
		$\lambda^k\rightarrow \lambda^*$ strongly in $L^2(Q)$.
\end{proof}

\section{Numerical results}\label{sec:numerical}

In this section, we solve a parabolic control constrained optimal control problem to validate the acceleration effectiveness of the primal-dual method (\ref{pd1})-(\ref{pd3}) with the enlarged step sizes (\ref{o_con_bound}). For the numerical discretization for all experiments, we employ the backward Euler finite difference scheme (with step size $\tau$) for the time discretization and the piecewise linear finite element method (with mesh size $h$) for the space discretization, respectively. Our codes were written in MATLAB R2020b and numerical experiments were conducted on a MacBook Pro with mac OS Monterey, Intel(R) Core(TM) i7-9570h (2.60 GHz), and 16 GB RAM.

We consider the following example:
\begin{equation}\label{ex1_Problem}
\begin{aligned}
\min_{u\in L^2(Q), y\in L^2(Q)} \quad  &\frac{1}{2}\|y-y_d\|_{L^2(Q)}^2+\frac{\alpha}{2}\|u\|_{L^2(Q)}^2+\theta(u),
\end{aligned}
\end{equation}
where $y$ and $u$ satisfy the following parabolic equation:
\begin{equation}\label{ex1_state}
\frac{\partial y}{\partial t}-\Delta y=f+u~ \text{in}~ \Omega\times(0,T), \quad
y=0~ \text{on}~ \Gamma\times(0,T),\quad y(0)=\varphi.
\end{equation}
Above, $\varphi\in L^2(\Omega)$, the function $f\in L^2(Q)$ is a source term that helps us construct the exact solution without affection to the numerical implementation. The nonsmooth term $\theta(u)$ is the indicator function of the admissible set (\ref{admissible_set}). We set $\Omega=(0,1)^2$, $\omega=\Omega$, $T=1$, $a=-0.5$, $b=0.5$ and
\begin{eqnarray*}
	&&y=(1-t)\sin\pi x_1\sin\pi x_2,~ q=\alpha (1-t)\sin 2\pi x_1\sin 2\pi x_2,~\varphi=\sin \pi x_1\sin \pi x_2, \\
	&&f=-u+\frac{dy}{dt}-\Delta y, ~ y_d=y+\frac{dq}{dt}+\Delta q,~ u=\min(-0.5,\max(0.5,-q/\alpha)).
\end{eqnarray*}
Then, it is easy to verify that $(u^*,y^*)^\top=(u, y)^\top$ is the optimal solution of  (\ref{ex1_Problem}). The problem (\ref{ex1_Problem}) has been discussed in, e.g. \cite{andrade2012multigrid,glowinski2022}.

For the purpose of numerical comparison, we also test the accelerated primal-dual (APD) method in \cite{chambolle2011first} and the inexact ADMM (InADMM) method in  \cite{glowinski2022}.
Numerical implementations of the InADMM follow all the settings in \cite{glowinski2022}, including the parameters settings, the solvers for the subproblems, and the stopping criteria. All algorithms to be tested are summarized below.
\begin{enumerate}
	\item [(1)] PD-C: The primal-dual method (\ref{pd1})--(\ref{pd3}) with the original convergence condition (\ref{convergence_condition});
	\item [(2)] PD-I: The primal-dual method (\ref{pd1})--(\ref{pd3}) with the enlarged step sizes (\ref{o_con_bound});
	\item [(3)] {APD($k$)}: The accelerated primal-dual method in \cite{chambolle2011first}, which adjusts the parameters every $k$ iterations;
	\item [(4)] {InADMM}: The inexact ADMM in \cite{glowinski2022} with CG inner iterations.
\end{enumerate}
The initial values for all primal-dual methods are set as $(u^0,p^0)^\top=(0,0)^\top$. For a prescribed tolerance $tol>0$, we terminate the iterations if
\begin{equation}\label{stopping}
	\max\Big\{\frac{\|u^{k+1}-u^k\|_{L^2(\mathcal{O})}}{\max\{1,\|u^k\|_{L^2(\mathcal{O})}\}},\frac{\|p^{k+1}-p^k\|_{L^2(Q)}}{\max\{1,\|p^k\|_{L^2(Q)}\}}\Big\}\leq tol.
\end{equation}

Recall that the upper bound of step sizes $1/\|S\|^2$ is enlarged by the factor $\frac{4+2\alpha r}{3}$. It is clear that the choice of $\alpha$ affects the value of $\frac{4+2\alpha r}{3}$ and thus has a further impact on the performance of PD-I. Intuitively, a relatively large $\alpha$ always leads to a large $\frac{4+2\alpha r}{3}$ and hence more likely improves the numerical efficiency. To validate this fact, we consider two different cases for problem (\ref{ex1_Problem}) in terms of the value of $\alpha$ in the following discussion.

\medskip
\noindent\textbf{Case I: $\alpha=10^{-3}$.} Concerning with the choices of $r$ and $s$ in all primal-dual methods, we note that, after the space-time discretization, one can estimate that $\|S\| = \|S^*\| \approx 0.05$ and this value is not affected
by the mesh sizes $\tau$ and $h$. According to (\ref{convergence_condition}), $r$ and $s$ should be chosen such that $r\cdot s<1/\|S^*S\|\approx 400$. Here, we choose $r=4\times 10^3$ and $s=1\times10^{-1}$ for PD-C. In addition, it follows from (\ref{o_con_bound}) that the upper bound of $r\cdot s$ can be enlarged by $\frac{4+2\alpha r}{3}=4$. We thus choose $r=4\times 10^3$ and $s=4\times 10^{-1}$ for PD-I. The parameters for all test algorithms are summarized in Table \ref{parameter_case1_ex1}.
\begin{table}[htpb]
	\centering
	\caption{Parameters for all primal-dual algorithms for Case I of Example 1}\label{parameter_case1_ex1}
	{\footnotesize
		\begin{tabular}{|c|c|c|c|c|c|c|}
			\hline
			Algorithm&Parameters\\
			\hline
			PD-C& $r=4\times 10^3, s=1\times 10^{-1}$\\
			\hline
			PD-I& $r= 4\times 10^3, s=4\times 10^{-1}$\\
			\hline
			{APD}(k)& $r_0=1\times 10^3, s_0=4\times 10^{-1};\forall k\geq0, \tau_k=\frac{1}{\sqrt{1+s_k}}, r_{k+1}=\frac{r_k}{\tau_k}; s_{k+1}=s_k\tau_k$\\
			\hline
		\end{tabular}
	}
\end{table}

The numerical results with $\tau=h={1}/{2^6}$ and $tol=10^{-5}$ are summarized in Table \ref{result_case1_ex1}.  We observe that PD-C is slower than InADMM, while PD-I is comparable to InADMM in terms of the total computational cost.
For APD, it is remarkable that implementing the adaptive step size selection strategy at every iteration is not efficient and it should be deliberately determined in practice, which is validated by the fact that APD(5) converges much faster than APD(1). We see that PD-I is more efficient than PD-C, APD(1), and APD(5). In particular, PD-I is 3 times faster than PD-C, which implies the superiority of the improved condition (\ref{o_con_bound}) to the original one (\ref{convergence_condition}).
\begin{table}[htpb]
	\centering
	\caption{Numerical comparisons for Case I of Example 1. ($\alpha=10^{-3 }, \tau=h={1}/{2^6}, tol=10^{-5}$)}\label{result_case1_ex1}
	{\footnotesize\begin{tabular}{|c|c|c|c|c|c|c|}
			\hline
			Algorithm&$Iter$&No. PDEs&$CPU$&$Obj$&$\|u^k-u^*\|$&$\|y^k-y^*\|$\\
			\hline
			PD-C&73&146 &14.3571 &$3.0742\times 10^{-4}$ &$2.3793\times 10^{-3}$ &$6.7691\times 10^{-5}$\\
			\hline
			PD-I& 24&48 &4.8687&$3.0742\times 10^{-4}$ &$2.3711\times 10^{-3}$ &$6.7512\times 10^{-5}$\\
			\hline
			{APD}(1)& 80&160&15.6948&$3.0741\times 10^{-4}$ &$2.6578\times 10^{-3}$ &$7.0752\times 10^{-5}$\\
			\hline
			{APD}(5)& 31&62 &5.6523&$3.0741\times 10^{-4}$ &$2.3913\times 10^{-3}$ &$6.7918\times 10^{-5}$\\
			\hline
		    {InADMM}& 19&50 &4.9675&$ 3.0741\times 10^{-4}$ &$2.3695\times 10^{-3}$ &$6.7459\times 10^{-5}$\\
			\hline
		\end{tabular}
	}
\end{table}

\begin{figure}[htpb]
	\caption{Numerical results at $t=0.25$ for Case I of Example 1. ($\alpha=10^{-3}, \tau=h={1}/{2^6}, tol=10^{-5}$)}\label{numerical_resultu_ex1}
	\centering
\subfigure[Computed control $u$]
	{\includegraphics[width=0.23\textwidth]{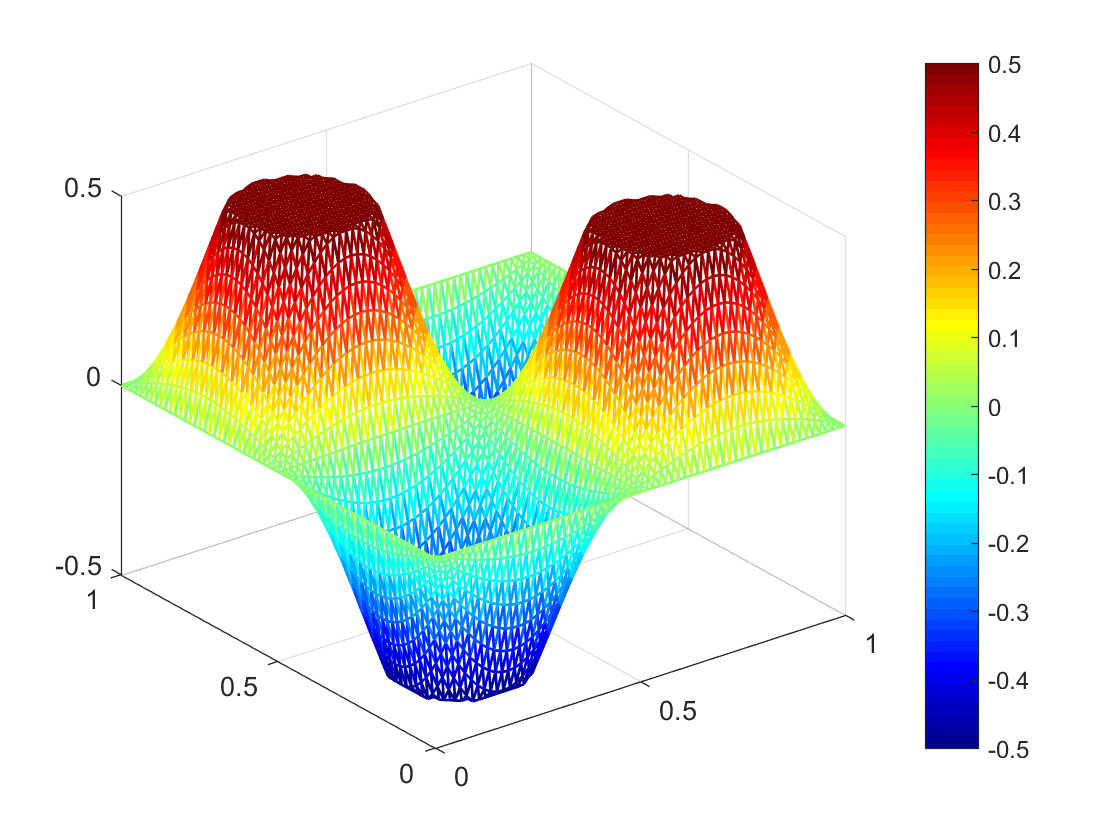}}
\subfigure[Computed state $y$]{\includegraphics[width=0.23\textwidth]{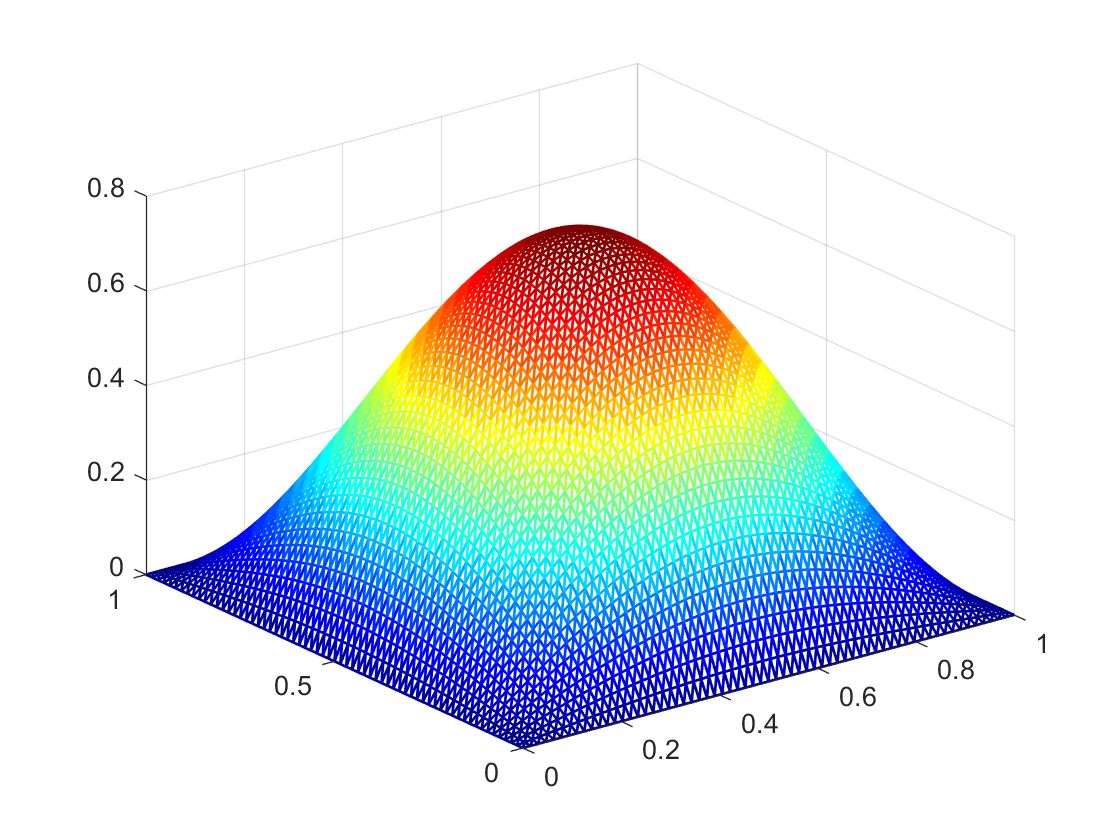}}
	\subfigure[Error $u-u^*$]	{\includegraphics[width=0.23\textwidth]{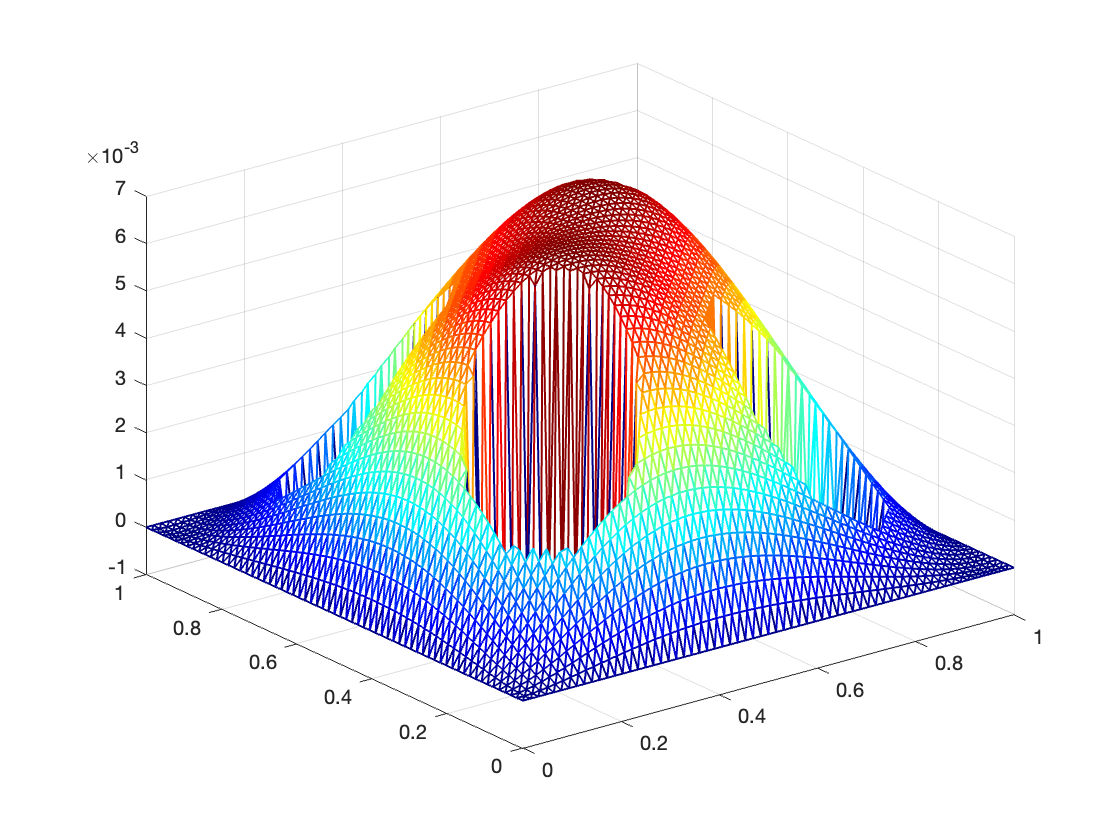}}
		\subfigure[Error $y-y^*$]{\includegraphics[width=0.23\textwidth]{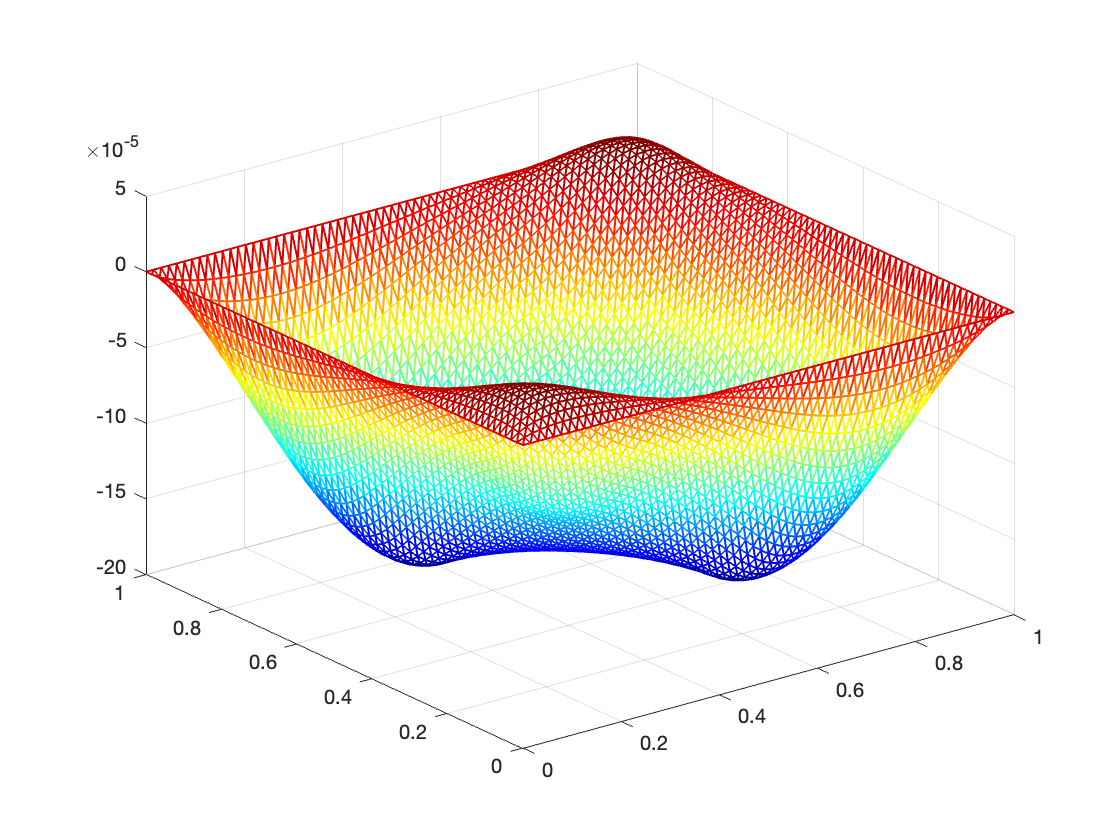}}
\end{figure}

\medskip

\noindent\textbf{Case II: $\alpha=10^{-5}$}. The parameters for all primal-dual methods are summarized in Table \ref{parameter_case2_ex1}. The numerical results with $\tau=h={1}/{2^6}$ and $tol=10^{-5}$ are presented in Table \ref{result_case2_ex1} and Figure \ref{numerical_results2_ex1}.
We observe from Table \ref{result_case2_ex1} that all primal-dual methods require less CPU time than that of InADMM. More specifically, although {InADMM} requires only 22 outer iterations, a total of 264 PDEs are required to be solved to promote the convergence. Compared with PD-C, PD-I can improve the numerical efficiency by $19.7\%$ and it is even faster than APD.  Here, we set $\alpha=10^{-5}$, which leads to the value of $\frac{4+2\alpha r}{3}$ relatively small, and hence compared with Case I, less numerical efficiency is improved by PD-I.
\begin{table}[h!]
	\centering
	\caption{Parameters for all primal-dual algorithms for Case II of Example 1.}\label{parameter_case2_ex1}
	{\footnotesize\begin{tabular}{|c|c|c|c|c|c|c|}
			\hline
			Algorithm&Parameters\\
			\hline
			PD-C& $r=4\times 10^3, s=1\times 10^{-1}$\\
			\hline
			PD-I& $r=5.6\times 10^3, s=1\times 10^{-1}$\\
			\hline
			{APD}(k)& $r_0=4\times 10^3, s_0=1\times 10^{-1};\forall k\geq0, \tau_k=\frac{1}{\sqrt{1+s_k}}, r_{k+1}=\frac{r_k}{\tau_k}; s_{k+1}=s_k\tau_k$\\
			\hline
		\end{tabular}
	}
\end{table}

\begin{table}[htpb]
	\centering
	\caption{Numerical comparisons for Case II of Example 1. ($\alpha=10^{-5}, \tau=h={1}/{2^6}, tol=10^{-5}$)}\label{result_case2_ex1}
	{\footnotesize
		\begin{tabular}{|c|c|c|c|c|c|c|}
			\hline
			Algorithm&$Iter$&No. PDEs&$CPU$&$Obj$&$\|u^k-u^*\|$&$\|y^k-y^*\|$\\
			\hline
			PD-C&122&244 &7.2705 &$3.5033\times 10^{-7}$ &$4.6747\times 10^{-3}$ &$8.6408\times 10^{-6}$\\
			\hline
			PD-I& 98&196 &5.8277&$3.5034\times 10^{-7}$ &$4.6715\times 10^{-3}$ &$8.6379\times 10^{-6}$\\
			\hline
			{APD}(1)& 118&236&7.0588&$3.5030\times 10^{-7}$ &$4.6731\times 10^{-3}$ &$8.5595\times 10^{-6}$\\
			\hline
			{APD}(5)& 106&212 &6.3541&$3.5036\times 10^{-7}$ &$4.6698\times 10^{-3}$ &$8.6628\times 10^{-6}$\\
			\hline
			{InADMM}& 22&264 &10.8675&$3.5035\times 10^{-7}$ &$4.6767\times 10^{-3}$ &$8.6088\times 10^{-6}$\\
			\hline
		\end{tabular}
	}
\end{table}

\begin{figure}[h!]
	\caption{Numerical results at $t=0.25$ for Case II of Example 1. ($\alpha=10^{-5}, \tau=h={1}/{2^6}, tol=10^{-5}$)}\label{numerical_results2_ex1}
	\centering
	\subfigure[Computed control $u$]
	{\includegraphics[width=0.23\textwidth]{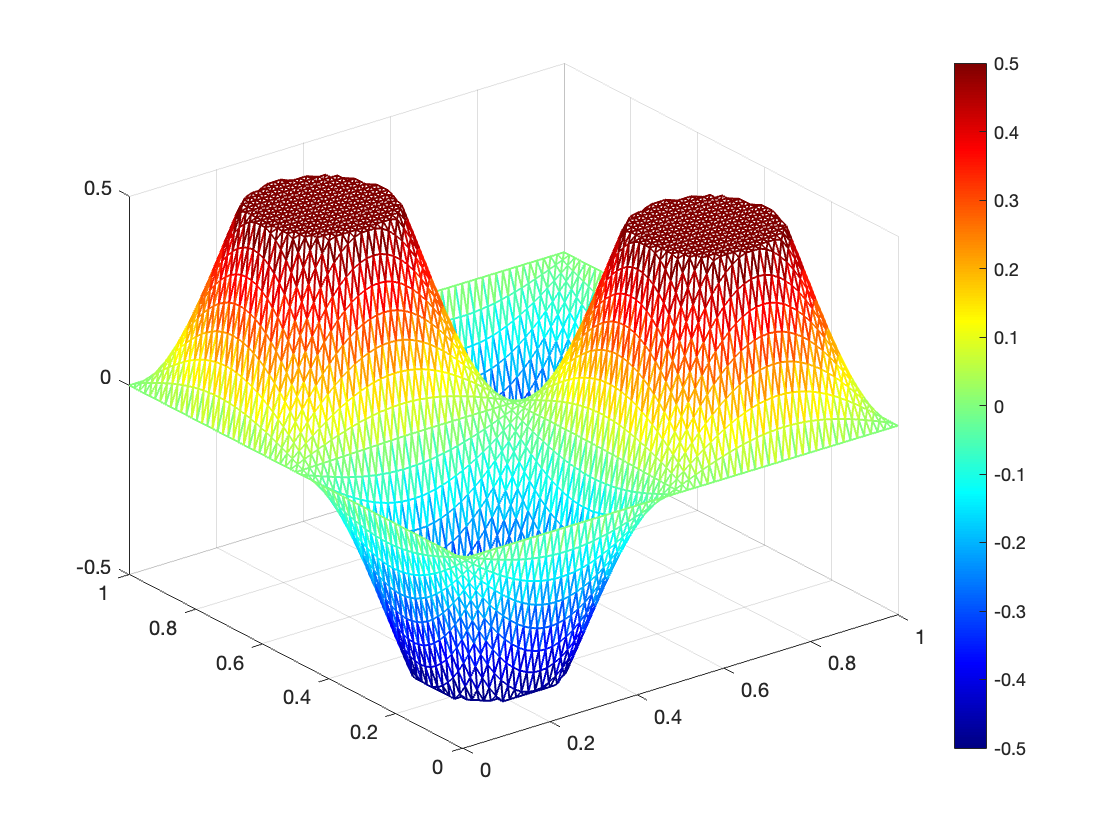}}
	\subfigure[Computed state $y$]{\includegraphics[width=0.23\textwidth]{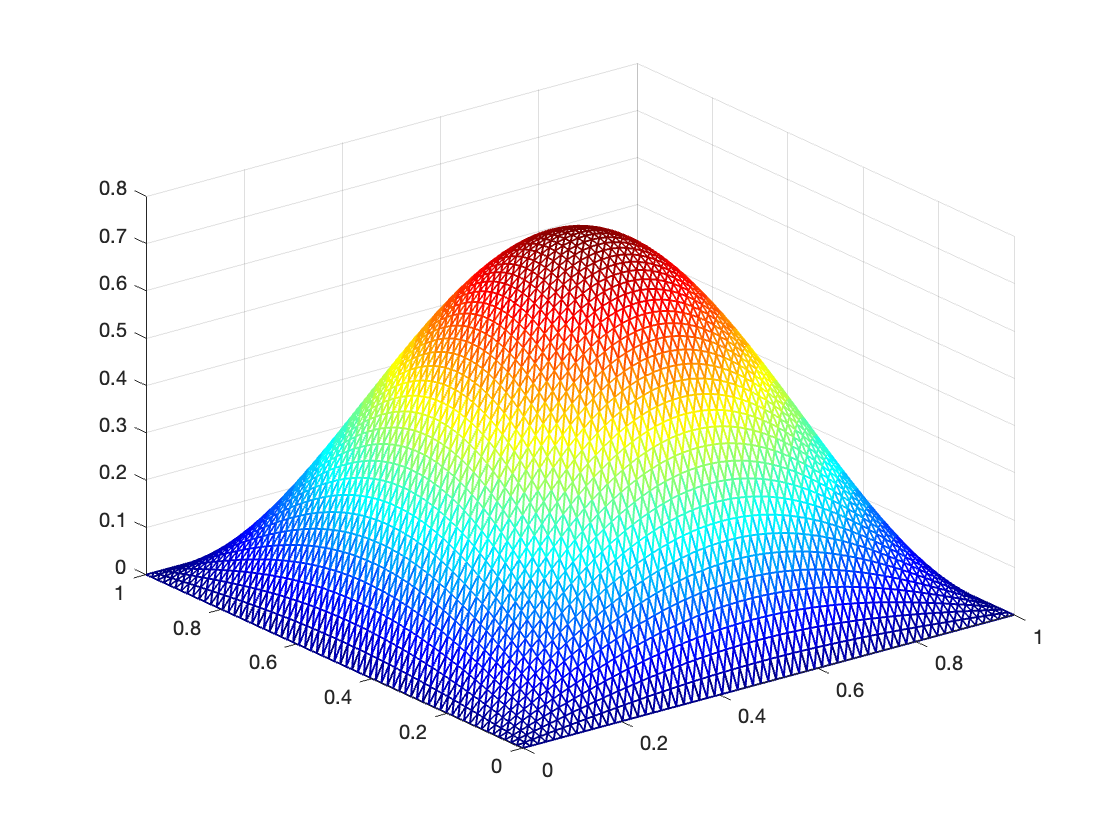}}
	\subfigure[Error $u-u^*$]	{\includegraphics[width=0.23\textwidth]{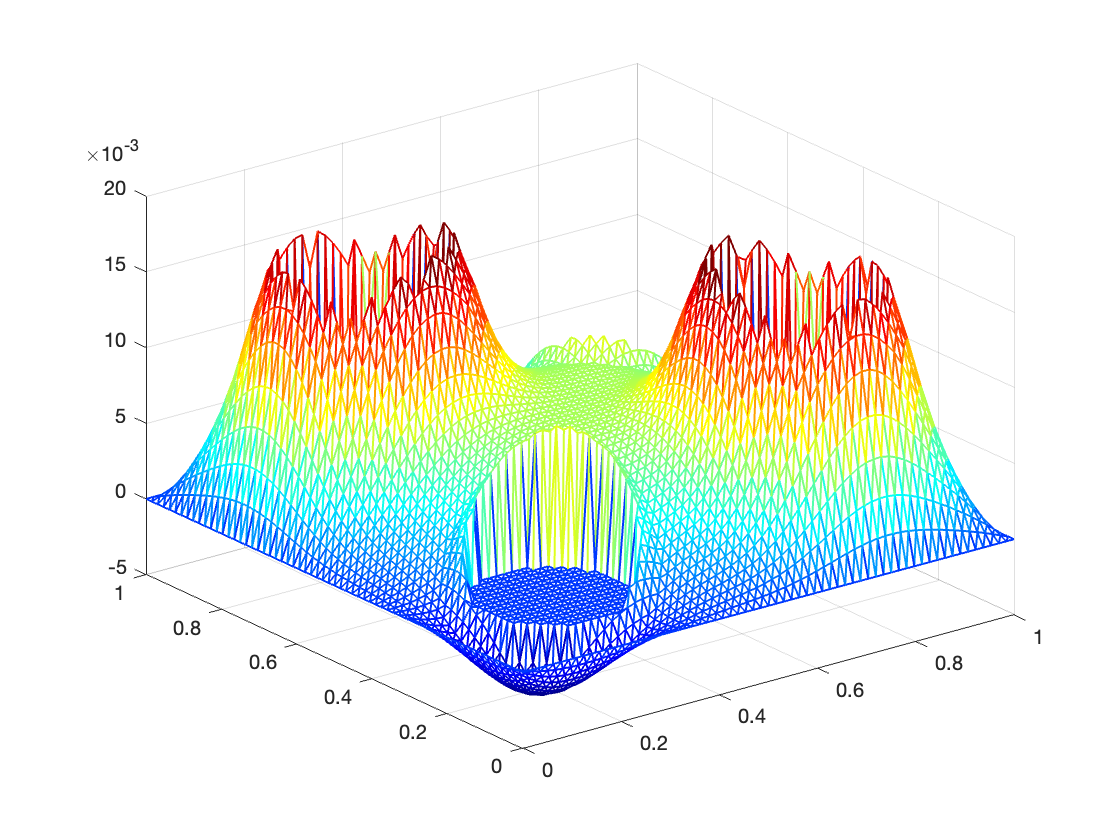}}
	\subfigure[Error $y-y^*$]{\includegraphics[width=0.23\textwidth]{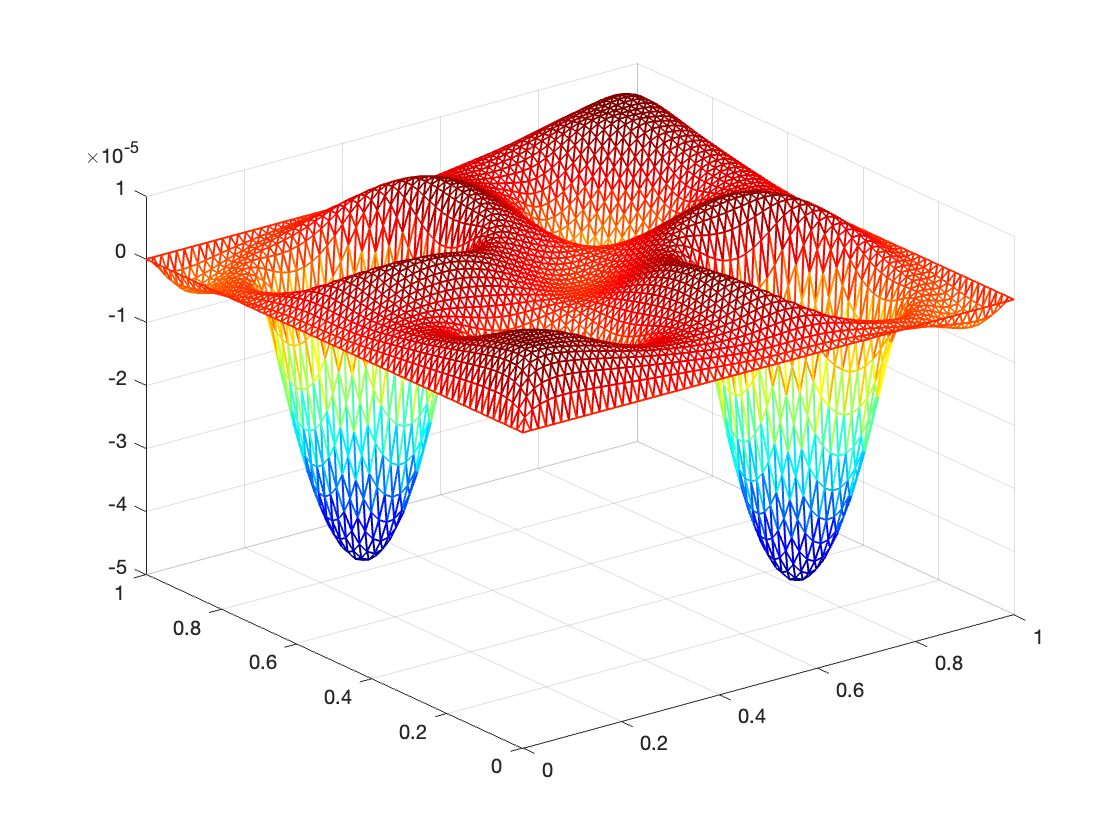}}
\end{figure}

Next, we recall that both PD-C and PD-I are described on the continuous level
and their convergence are analyzed in function spaces. Hence, mesh-independent property of these algorithms can be expected in practice, which means that the convergence behavior is independent of the fineness of the discretization. We test PD-C and PD-I with $\alpha=10^{-3}$ and  $\tau=h={1}/{2^i},i=4,\cdots,9$, and report the iteration numbers in Table \ref{meshindependent_case2_ex1}, from which mesh-independent properties of PD-C and PD-I can be observed.
\begin{table}[h!]
	\centering
	\caption{Iteration numbers w.r.t  different mesh sizes for Example 1.}\label{meshindependent_case2_ex1}
	{\footnotesize\begin{tabular}{|c|c|c|c|c|c|c|}
			\hline
			Mesh size&$1/{2^4}$&$1/{2^5}$&$1/{2^6}$&$1/{2^7}$&$1/{2^8}$&$1/{2^9}$\\
			\hline
			PD-C&73 &73& 73 & 73 & 73& 73\\
			\hline
			PD-I& 25& 25& 24& 24& 23 & 23\\
			\hline
		\end{tabular}
	}
\end{table}

Finally, in Table \ref{err_ex1}, we report the $L^2$-error for the iterate ($u$, $y$) obtained by PD-I for various values of $h$ and $\tau$. For succinctness, we only give the results for the case where $\alpha=10^{-5}$ and $tol= 10^{-5}$. It is clear from Table \ref{err_ex1} that, when  PD-I is applied to the problem (\ref{ex1_Problem}), the iterative accuracy is sufficient and the overall error of $u$ and $y$ are both dominated by the discretization error.

\begin{table}[h!]
	\setlength{\abovecaptionskip}{0pt}
	\setlength{\belowcaptionskip}{3pt}
	\centering
	\caption{Numerical errors of PD-I for Example 1. ($\alpha=10^{-3},tol=10^{-5}$)}\label{err_ex1}
	{\small\begin{tabular}{|c|c|c|c|c|c|}
			\hline
			error&$h=\tau=2^{-5}$&$h=\tau=2^{-6}$&$h=\tau=2^{-7}$&$h=\tau=2^{-8}$\\
			\hline
			$\|u-u^*\|_{L^2(Q)}$&$1.8404\times 10^{-2}$ & $4.6715\times 10^{-3}$&$1.1815\times 10^{-3}$ &$3.2924\times 10^{-4}$\\
			\hline
			$\|y-y^*\|_{L^2(Q)}$& $ 3.6458\times 10^{-5}$ & $  8.6370\times 10^{-6}$& $ 2.1690\times 10^{-6}$&$5..8059\times 10^{-7}$\\
			\hline
		\end{tabular}
	}
\end{table}

\section{Extension: A sparse elliptic optimal control problem}\label{se:se}
In previous sections, we focus on the parabolic optimal control problem (\ref{optimal_control})--(\ref{admissible_set}) to expose our main ideas more clearly. As mentioned in the introduction, various optimal control problems can be covered by the model (\ref{Basic_Problem}) and all previous discussions can be easily extended to them. In this section, we showcase by a sparse elliptic optimal control problem to delineate how to extend the primal-dual method (\ref{pd1})-(\ref{pd3}) with the enlarged step sizes (\ref{o_con_bound}) to other optimal control problems. Some notations and discussions analogous to previous ones are not repeated for succinctness.

Let us consider the following sparse elliptic optimal control problem:
\begin{equation}\label{modelproblem}
\underset{u\in L^2(\mathcal{O}),y\in H_0^1(\Omega)}{\min}~ J(y,u)=\frac{1}{2}\|y-y_d\|_{L^2(\Omega)}^2+\frac{\alpha}{2}\|u\|_{L^2(\Omega)}^2+{\mu}\|u\|_{L^1(\Omega)}+I_{U_{ad}}(u),
\end{equation}
where $y$ and $u$ satisfy the following state equation:
\begin{equation}\label{state_equation}
	-\Delta y=u ~ \text{in}~ \Omega, \quad
	y=0~ \text{on}~ \Gamma.
\end{equation}
In (\ref{modelproblem})-(\ref{state_equation}), {$\Omega\subset \mathbb{R}^d(d\ge 1)$ is a convex polyhedral domain} with boundary $\Gamma:=\partial\Omega$, $y_d\in L^2(\Omega)$ is a given target, and the constants $\alpha>0$ and $\mu>0$ are regularization parameters.
We denote by $I_{U_{ad}}(\cdot)$ the indicator function of the admissible set
$
	U_{ad}:=\{u\in L^\infty(\Omega)| a\leq u(x)\leq b, ~\text{a.e.~in}~ \Omega \}\subset L^2(\Omega),
$
where $a,b \in L^2(\Omega)$ with $a < 0 < b$ almost everywhere. Due to the presence of the nonsmooth $L^1$-regularization term,  the optimal control of (\ref{modelproblem}) has small support  \cite{stadler2009elliptic,wachsmuth2011}. Because of this special structural property, such problems capture important applications in various fields such as optimal actuator placement \cite{stadler2009elliptic} and impulse control \cite{ciaramella2016}.

\subsection{Primal-dual method for (\ref{modelproblem})-(\ref{state_equation})}

Similar to what we have done for (\ref{optimal_control})--(\ref{admissible_set}), implementing the primal-dual method in \cite{chambolle2011first} to (\ref{modelproblem})-(\ref{state_equation}) yields the following iterative scheme:
\begin{equation}\label{ue_k+1}
~u^{k+1}=P_{U_{ad}}\left(\mathbb{S}_{\frac{\mu r}{\alpha r+1}}\left(\frac{u^k-rS^*{p}^k}{\alpha r+1}\right)\right),\quad
p^{k+1}=\left(S(2u^{k+1}-u^k)+\frac{1}{s}p^k-y_d\right)/(1+\frac{1}{s}),
\end{equation}
where $S:L^2(\Omega)\rightarrow L^2(\Omega)$ such that $y=Su$ is the solution operator associated with the elliptic state equation (\ref{state_equation}), $S^*: L^2(\Omega)\rightarrow L^2(\Omega)$ is the adjoint operator of $S$, $P_{U_{ad}}(\cdot)$ denotes the projection onto the admissible set $U_{ad}$, namely, $P_{U_{ad}}(v)(x) := \max\{a, \min\{v(x), b\}\}$ a.e. in $\Omega$, $\forall v\in L^2(\Omega)$, and $\mathbb{S}$ is the Shrinkage operator defined by
$$
\mathbb{S}_\zeta(v)(x) = \text{sgn}(v(x)) (|v(x)|-\zeta)_{+}~\text{a.e. in }~\Omega,
$$
where $\zeta$ a positive constant, ``$\text{sgn}$" is the sign function, and $(\cdot)_{+}$ denotes the positive part. Under the condition (\ref{convergence_condition}) or (\ref{o_con_bound}), both the PD-C and PD-I can be proposed for problem (\ref{modelproblem})-(\ref{state_equation}), and convergence analysis can simply follow the results in Section \ref{sec:convergence_primal_dual} directly. At each iteration of (\ref{ue_k+1}), the main computation consists of solving the state equation (\ref{state_equation}) to compute $S(2u^{k+1}-u^k)$, and the adjoint equation
$
-\Delta q^k=p^k~\text{in}~\Omega,~q=0~\text{on}~ \Gamma,
$
to compute $q^k=S^*p^k$.


\subsection{Numerical results}

In this subsection, we report some numerical results to validate the efficiency of the primal-dual method (\ref{ue_k+1}) for solving (\ref{modelproblem})-(\ref{state_equation}).

\medskip
\noindent{\textbf{Example 2.}} We consider the example given in \cite{stadler2009elliptic}. To be concrete, we set $\Omega=(0,1)\times(0,1)$,  $a=-30, b = 30$, and $y_d =
\frac{1}{6}e^{2x_1}\sin(2\pi x_1)\sin(2\pi x_2)$ in (\ref{modelproblem})-(\ref{state_equation}). In all numerical experiments, the numerical discretization is implemented by the finite element method described in \cite{wachsmuth2011}. We test the PD-C and PD-I for two cases in terms of the choice of $\alpha$. The initial value is set as $(u^0,p^0)^\top=(0,0)^\top$, and all algorithms are terminated if (\ref{stopping}) holds with $tol=10^{-5}$.

First, we set $\mu=5\times10^{-3}$ and $\alpha= 1\times10^{-3}$. The parameters are selected as those listed in Table \ref{parameter_case1_ex1}. We summarize the numerical results in Table \ref{tab:meshindependent_case1_ex3}. It is clear that all algorithms are robust to the mesh size, and mesh-independent convergence can be observed. The PD-I improves the numerical efficiency significantly. The numerical results $u$ and $y$ obtained by PD-I with $h={1}/{2^6}$ are reported in Figure \ref{numerical_resultu_case1_ex3}. As expected, we note that  $u = 0$ on a relatively
large part of $\Omega$ due to the presence of the regularization term $\mu\|u\|_{L^1(\Omega)}$.
\begin{table}[h!]
	\centering
	\caption{Iteration numbers w.r.t different mesh sizes for Example 2 when $\alpha=10^{-3}, \mu=5\times10^{-3}$.}\label{tab:meshindependent_case1_ex3}
	{\footnotesize\begin{tabular}{|c|c|c|c|c|c|c|}
			\hline
			Mesh size&$1/{2^4}$&$1/{2^5}$&$1/{2^6}$&$1/{2^7}$&$1/{2^8}$&$1/{2^9}$\\
			\hline
			PD-C&98 &97& 97& 97& 97& 97\\
			\hline
			PD-I&33& 33&33 &33 & 33& 33\\
			\hline
		\end{tabular}
	}
\end{table}

\begin{figure}[h!]
	\caption{ Numerical results for Example 2 when $\mu=5\times10^{-3}$.}\label{numerical_resultu_case1_ex3}
	\centering
	\subfigure[Control $u$ ($\alpha=10^{-3}$)]
	{\includegraphics[width=0.232\textwidth]{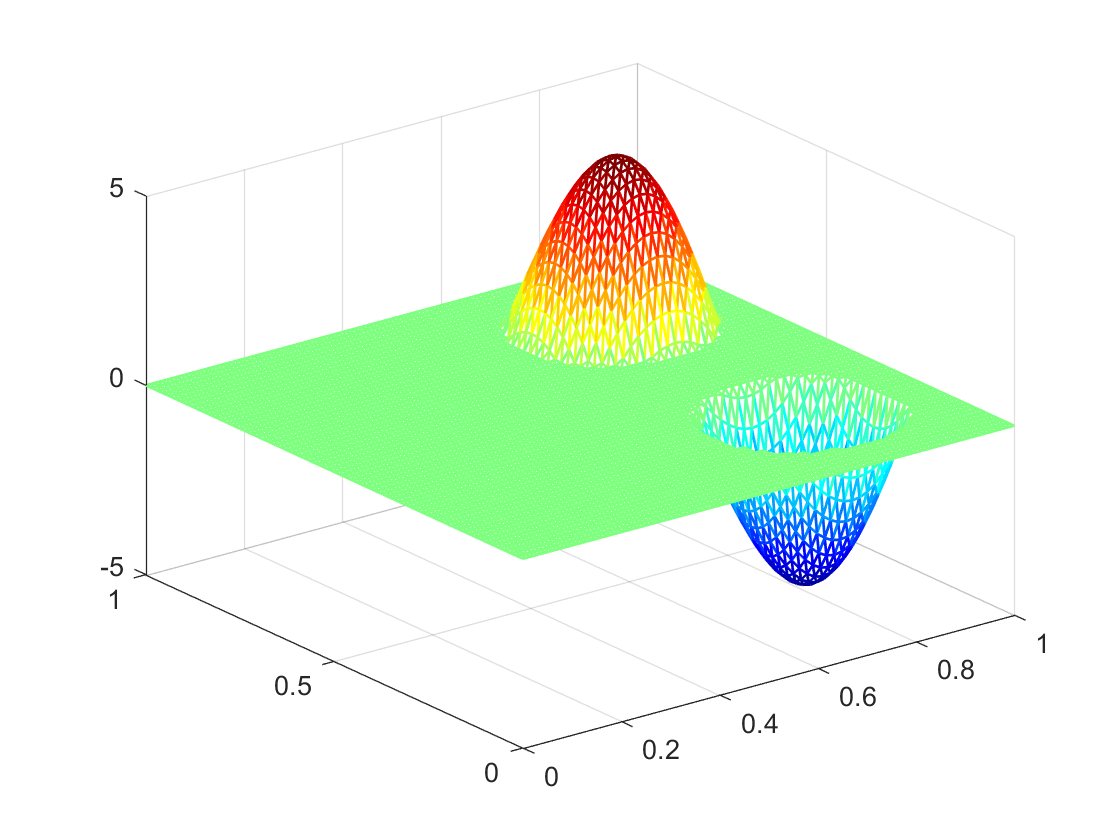}}
	\subfigure[State $y$ ($\alpha=10^{-3}$)]{\includegraphics[width=0.232\textwidth]{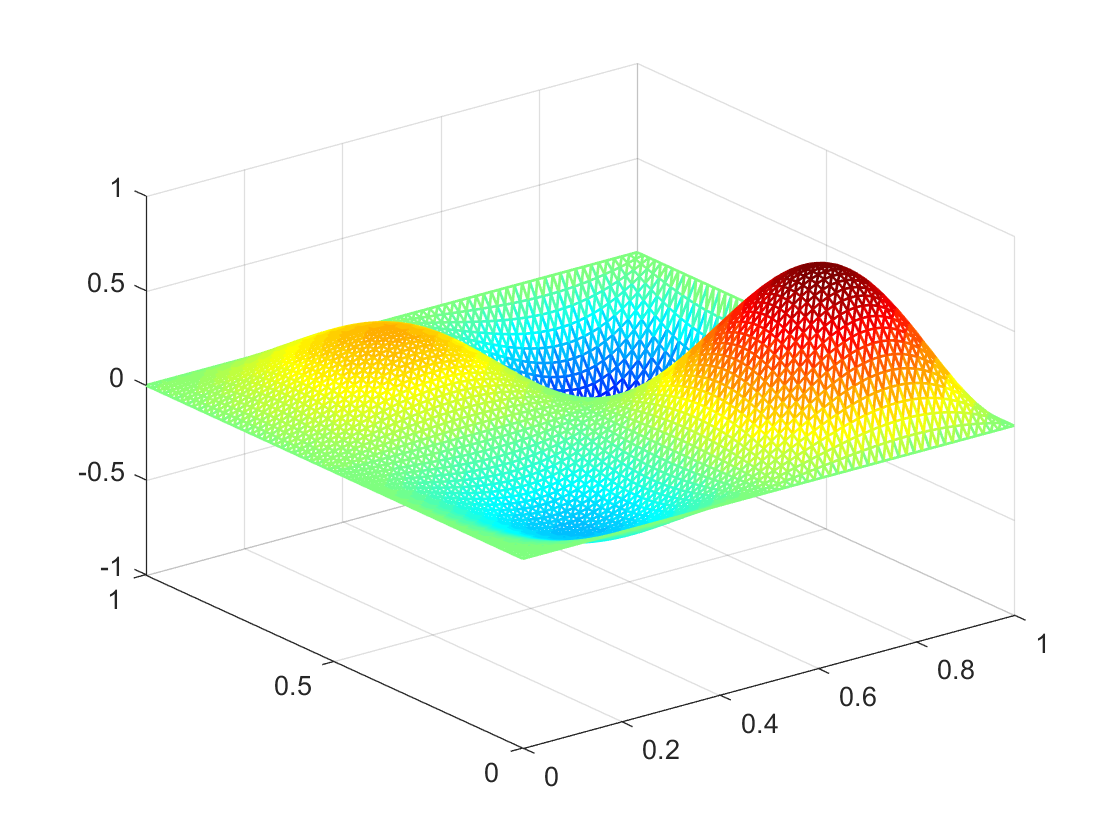}}
		\subfigure[Control $u$ ($\alpha=10^{-5}$)]	{\includegraphics[width=0.232\textwidth]{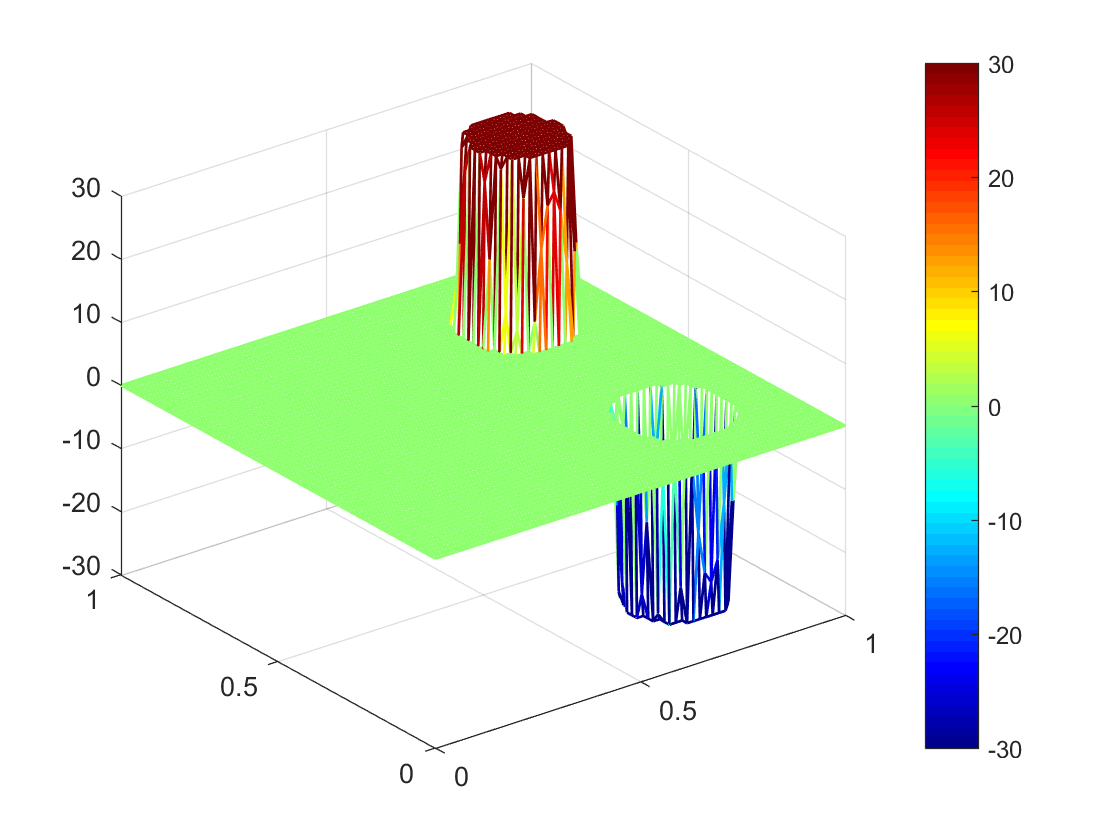}}
\subfigure[State $y$ ($\alpha=10^{-5}$)]{\includegraphics[width=0.232\textwidth]{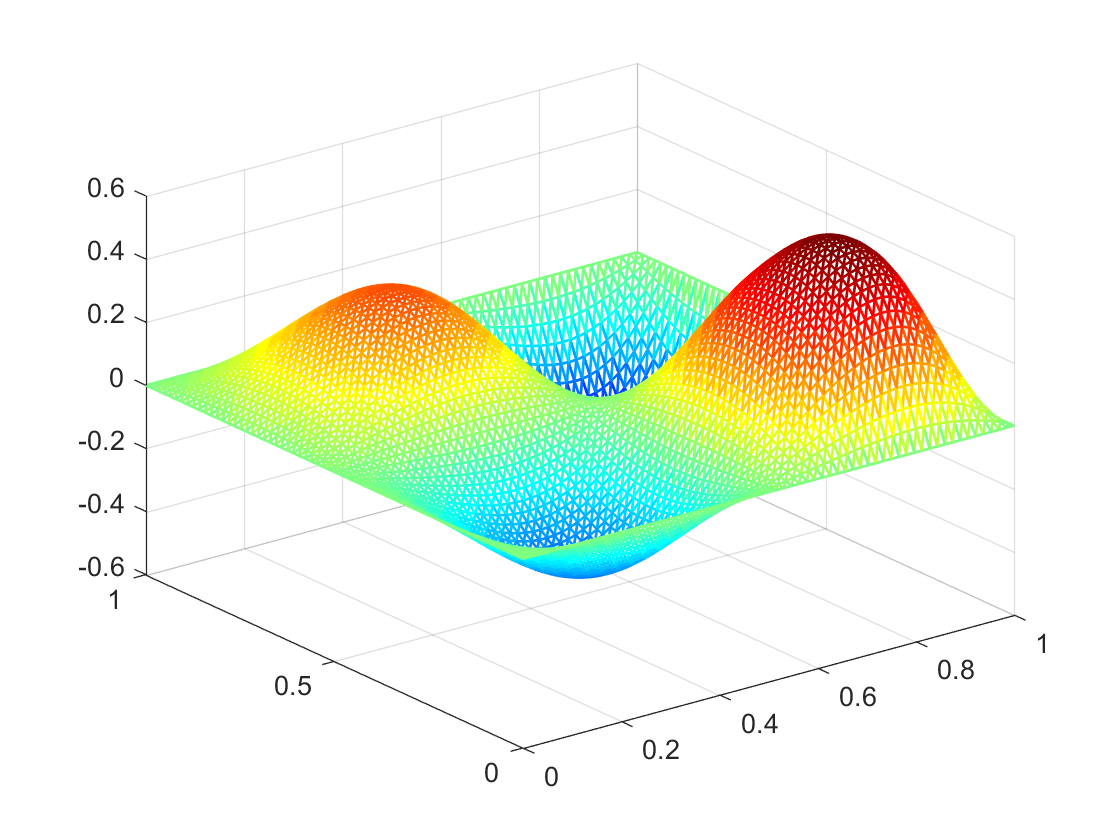}}
\end{figure}

Second, we still set $\mu=5\times10^{-3}$ but $\alpha= 1\times10^{-5}$. For this case, the parameters are selected as those listed in Table \ref{parameter_case2_ex1}. The numerical results of all test algorithms with respect to different mesh sizes are presented in Table \ref{tab:meshindependent2_ex3}. We observe from these results that the performance of all algorithms are robust to the mesh sizes, while PD-I improves the numerical efficiency of PD-C sharply. The numerical results $u$ and $y$ obtained by PD-I with $h={1}/{2^6}$ are reported in Figure \ref{numerical_resultu_case1_ex3}.

\begin{table}[h!]
	\centering
	\caption{Iteration numbers w.r.t  different mesh sizes for Example 2 when $\alpha=10^{-5}, \mu=5\times10^{-3}$.}\label{tab:meshindependent2_ex3}
	{\footnotesize\begin{tabular}{|c|c|c|c|c|c|c|}
			\hline
			Mesh size&$1/{2^4}$&$1/{2^5}$&$1/{2^6}$&$1/{2^7}$&$1/{2^8}$&$1/{2^9}$\\
			\hline
			PD-C&114 &147& 161& 165& 169& 171\\
			\hline
			PD-I&82& 106&116 &119 & 121& 123\\
			\hline
		\end{tabular}
	}
\end{table}

Next, we study the effectiveness of $\mu$ on the performance of PD-C and PD-I. For this purpose, we implement both of them to Example 3 with different $\mu$ and $\alpha=10^{-3}$.
The results  are reported in Table \ref{tab:beta_ex3} and the computed optimal controls are depicted in Figure \ref{numerical_resultu_beta}. These results indicate that all algorithms are robust to the values of $\mu$. Moreover, it was shown in \cite{stadler2009elliptic} that, as $\mu$ increases, the size of the nonzero region of $u$ decreases, and when $\mu$ is sufficiently large, $u$ is zero on the whole $\Omega$.   From Figure \ref{numerical_resultu_beta}, it is easy to see that the nonzero part of $u$ decreases as $\mu$ increases, which coincides with the results in \cite{stadler2009elliptic}.

\begin{table}[h!]
	\centering
	\caption{\small Numerical comparisons w.r.t different $\mu$ for Example 2 when $\alpha=10^{-3}.$}
	\label{tab:beta_ex3}
	{\footnotesize \begin{tabular}{|c|c|c|c|c|c|c|}
			\hline
			\multirow{2}{*}{$\mu$}&
			\multicolumn{2}{c|}{PD-C}&\multicolumn{2}{c|}{ PD-I}\cr\cline{2-5}
			&\footnotesize{Iter}&\footnotesize{$\|y-y_d\|_{L^2(\Omega)}$}&\footnotesize{Iter}&\footnotesize{$\|y-y_d\|_{L^2(\Omega)}$}\cr
			\hline
			$0$ &87&$2.4963\times 10^{-1}$&30&$2.4963\times 10^{-1}$\cr\hline
			
			
			$5\times10^{-4}$&88&$2.5356\times 10^{-1}$&31&$2.5356\times 10^{-1}$\cr\hline
			
			$3\times10^{-3}$&93&$2.7034\times 10^{-1}$&32&$2.7034\times 10^{-1}$ \cr\hline
			
	
		    $2\times10^{-2}$&97&$2.9018\times 10^{-1}$&32&$2.9018\times 10^{-1}$ \cr\hline
		\end{tabular}
	}
\end{table}
\begin{figure}[h!]
	\caption{ Numerical controls $u$ (from above view) w.r.t different $\mu$ for Example  2 when $\alpha=10^{-3}$. (noz=$\frac{|u\neq 0|}{|\Omega|}$)}\label{numerical_resultu_beta}
	\centering
	\subfigure[$\mu=0, noz=1$]{
	\includegraphics[width=0.23\textwidth]{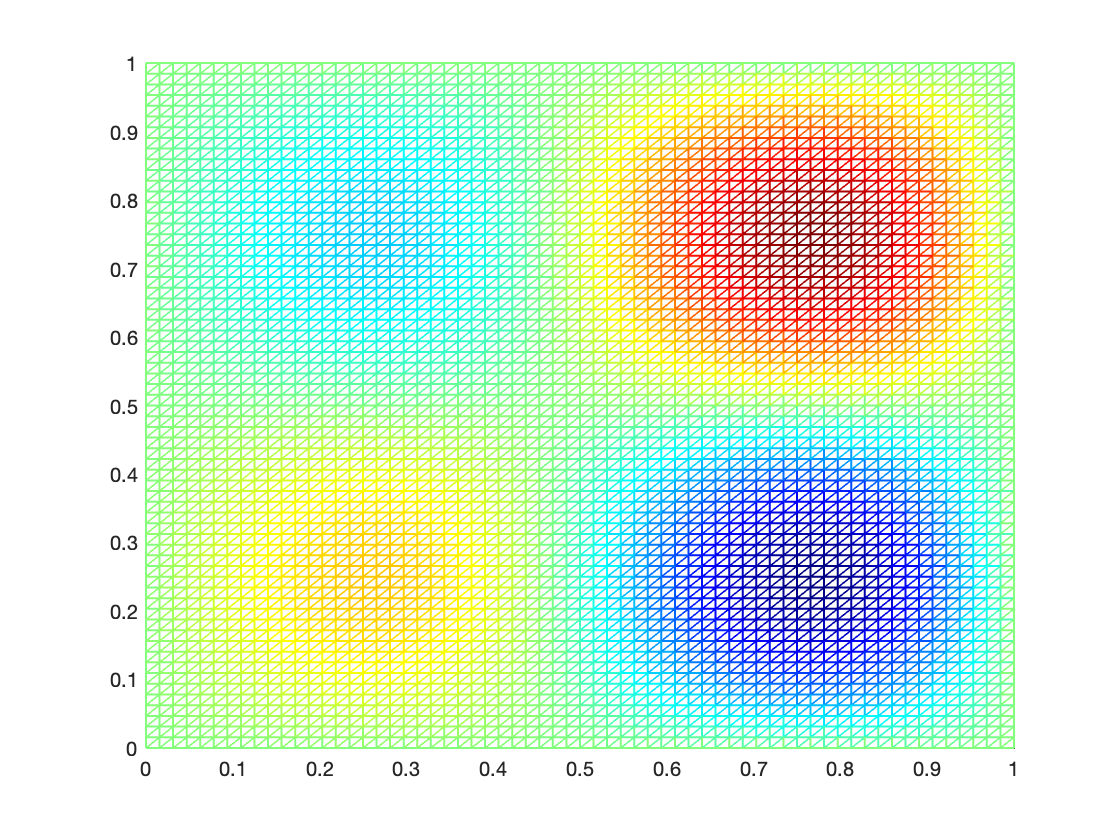}}
	\subfigure[$\mu=5\times10^{-4},noz=0.83$]{
	\includegraphics[width=0.23\textwidth]{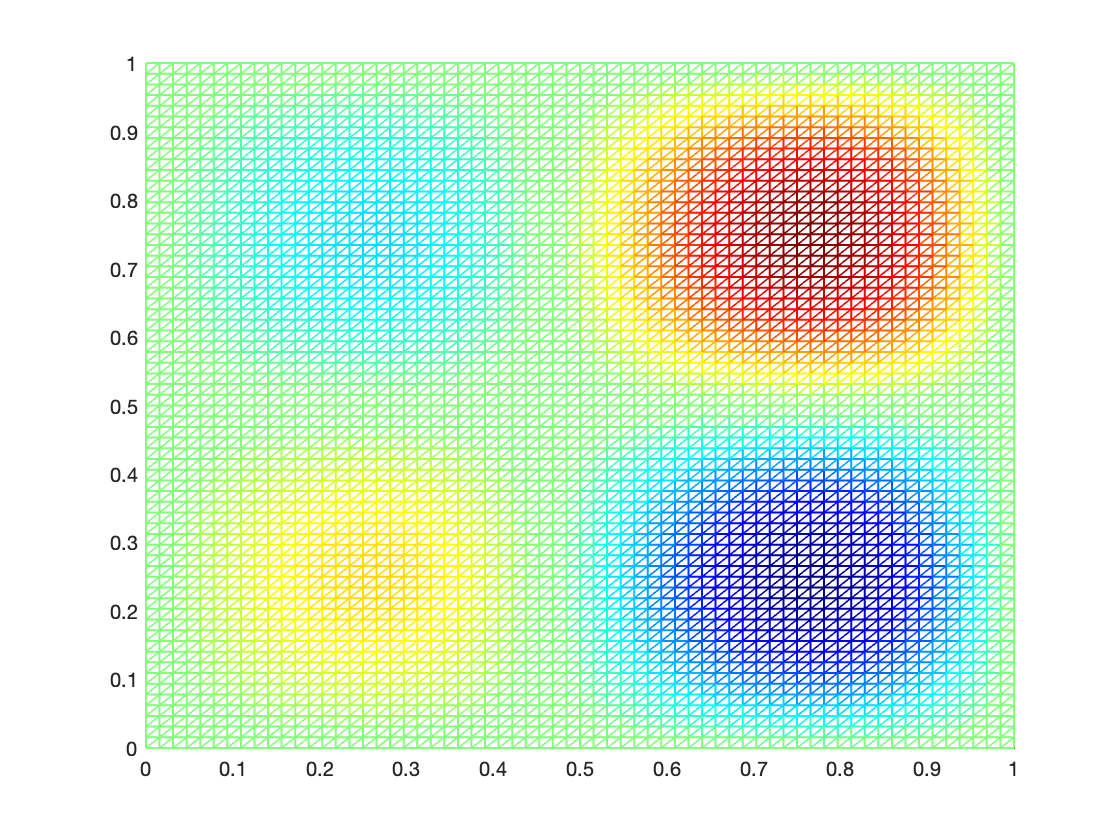}}
	\subfigure[$\mu=3\times10^{-3},noz=0.32$]{
	\includegraphics[width=0.23\textwidth]{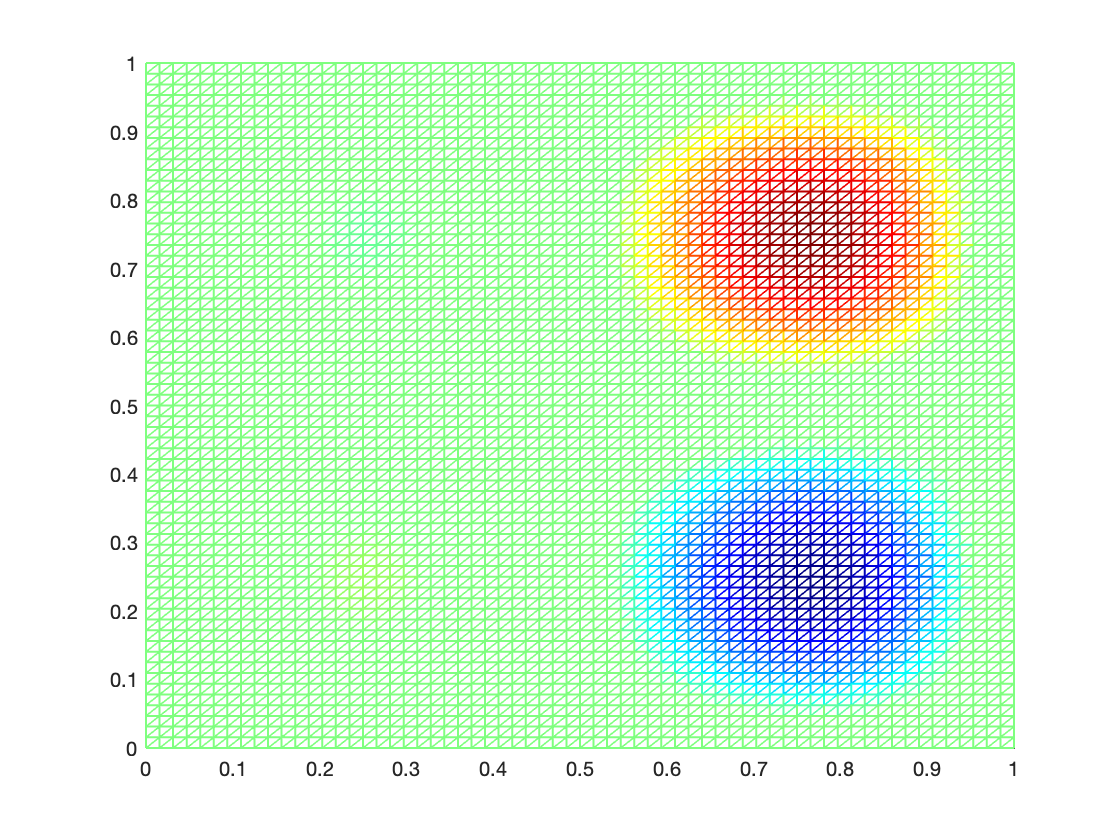}}
	\subfigure[$\mu=2\times10^{-2},noz=0$]{
	\includegraphics[width=0.23\textwidth]{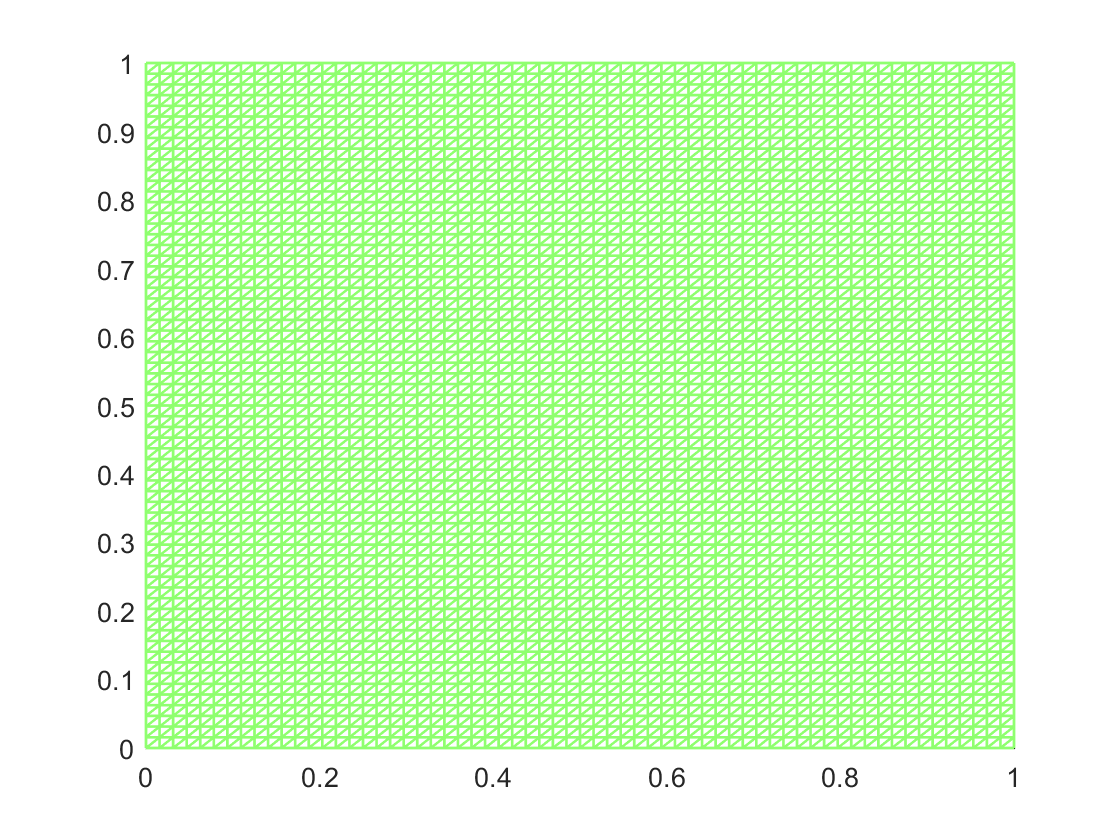}}
\end{figure}

\section{Implementation of the primal-dual method with operator learning (\ref{pdol})}\label{se: pdol}
In this section, we shall delineate the implementation of the primal-dual method with operator learning (\ref{pdol}) and specify some particular algorithms. To this end, we assume that $\mathcal{O}=Q$ in (\ref{optimal_control})-(\ref{admissible_set}) to simplify the notation. Then,  the central concern is constructing two surrogates $y=\mathcal{S}_{\theta_s}(u)$ and $q=\mathcal{S}_{\theta_a}(p)$ for the state equation $y=Su$ and its adjoint $q=S^*q$, respectively, by some operator learning methods.

We first elaborate on the main ideas of operator learning.
For this purpose, let $G$ be the solution operator of a generic PDE defined on the domain $\mathcal{D}$ which takes an input function $u$, and $y=G(u)$ be the solution of the PDE. Operator learning aims at approximating  $G$ with a neural network  $\mathcal{G}_{\theta}$. For any point $z\in\mathcal{D}$, $G(u)(z)$ is a real number (the value of $y$ at $z$). The neural network $\mathcal{G}_{\theta}$ takes inputs consisting of $u$ and $z$, and outputs the value $\mathcal{G}_{\theta}(u)(z)$. Suppose that we have a data set $\{G(u_i)(z_j)\}_{1\leq i\leq N_1, 1\leq j\leq N_2}$ of different input functions $\{u_i\}$ and points $\{z_j\}$. Then, the neural network is trained by solving
$$
\theta^*=\min_{\theta}\|\mathcal{G}_{\theta}(u)(z)-G(u)(z)\|_{L^2(\mathcal{D})}^2.
$$
Operator learning provides a surrogate model $y=\mathcal{G}_{\theta^*}(u)$ for $y=G(u)$. Several operator learning methods, such as the DeepONets \cite{lu2021learning}, the physics-informed DeepONets \cite{wang2021}, the FNO \cite{li2020FNO}, the GNO \cite{li2020GNO}, and the LNO \cite{cao2023}, have been recently proposed in the PDE literature. In the following discussion, to expose our main ideas clearly, we focus on the DeepONets \cite{wang2021} to elaborate on the implementation of (\ref{pdol}). Other operator learning methods can be applied in a similar way.

\begin{figure}[htpb]
	\caption{Workflow of the DeepONets. Source: the figure is adapted from \cite{lu2021learning}.}\label{fig:deepOnet}
	\begin{tikzpicture}[global scale =0.6]
			\hspace{1.6cm}
	\centering		
	\node at(0,11.5)     {\textbf{a}};
	\node at(2,11.2)     {\textbf{Stacked DeepONet}};
	\node at(0,9)     (7) {$u$};
	\node at(1.5,9)     (8) [zbox]{$\begin{aligned}
			u(x_1)\\u(x_2)\\ \cdots\\u(x_m)
		\end{aligned}$};
	\node at(4.5,10)    (9)[zbox,fill=green!20]{Branch net$_1$};
	\node at(4.5,9)    (10)[zbox,fill=green!20]{Branch net$_2$};
	\node at(4.5,8)        {$\cdots$};
	\node at(4.5,7.3)  (11)[zbox,fill=green!20]{Branch net$_n$};
	\node at(4.5,4.5)  (12)[zbox,fill=green!20]{Trunk net};
	\node at(7,9)(13)[rectangle, minimum width =25pt, minimum height =110pt, inner sep=5pt,draw=black]  {} ;
	\node at(7,4.5)(14)[rectangle, minimum width =25pt, minimum height =110pt, inner sep=5pt,draw=black]  {} ;
	\node at(7.05,10.5) (15)[circle, minimum width =17pt, minimum height =17pt, inner sep=0.1pt, fill=blue!20,draw=black]{$b_1$};
	\node at(7.05,9.5) (16)[circle, minimum width =17pt, minimum height =17pt, inner sep=0.1pt, fill=blue!20,draw=black]{$b_2$};
	\node at(7.05,8.5)          {$\cdots$};
	\node at(7.05,7.5) (17)[circle, minimum width =17pt, minimum height =17pt, inner sep=0.1pt, fill=blue!20,draw=black]{$b_n$};
	\node at(7.05,6)   (18)[circle, minimum width =17pt, minimum height =17pt, inner sep=0.1pt, fill=blue!20,draw=black]{$t_1$};
	\node at(7.05,5)   (19)[circle, minimum width =17pt, minimum height =17pt, inner sep=0.1pt, fill=blue!20,draw=black]{$t_2$};
	\node at(7.05,4)          {$\cdots$};
	\node at(7.05,3)   (20)[circle, minimum width =17pt, minimum height =17pt, inner sep=0.1pt, fill=blue!20,draw=black]{$t_n$};
	\node at(8.3,7)    (21)[circle, minimum width =17pt, minimum height =17pt, inner sep=0.1pt, fill=blue!20,draw=black]{$\times$};
	\node at(9.5,7)   (22){$\mathcal{G}_{\theta}(u)(z)$};
	\node at(2,4.5)   (23){$z$};
	
	\draw[->] (7) --(8);
	\draw[->] (8) --(9);
	\draw[->] (8) --(10);
	\draw[->] (8) --(11);
	\draw[->] (23) --(12);
	\draw[->] (9) --(15);
	\draw[->] (10) --(16);
	\draw[->] (11) --(17);
	\draw[->] (12) --(18);
	\draw[->] (12) --(19);
	\draw[->] (12) --(20);
	\draw[->] (13) --(21);
	\draw[->] (14) --(21);
	\draw[->] (21) --(8.85,7);

	\hspace{0.2cm}

	\node at(10.5,11.5)     {\textbf{b}};
	\node at(12.5,11.2)     {\textbf{Unstacked DeepONet}};
	\node at(10.5,9)    (24) {$u$};
	\node at(12,9)      (25) [zbox]{$\begin{aligned}
			u(x_1)\\u(x_2)\\ \cdots\\u(x_m)
		\end{aligned}$};
	\node at(15,9)      (26)[zbox,fill=green!20]{Branch net};
	\node at(15,4.5)    (27)[zbox,fill=green!20]{Trunk net};
	\node at(17.5,9)    (28)[rectangle, minimum width =25pt, minimum height =110pt, inner sep=5pt,draw=black]  {} ;
	\node at(17.5,4.5)  (29)[rectangle, minimum width =25pt, minimum height =110pt, inner sep=5pt,draw=black]  {} ;
	\node at(17.55,10.5) (30)[circle, minimum width =17pt, minimum height =17pt, inner sep=0.1pt, fill=blue!20,draw=black]{$b_1$};
	\node at(17.55,9.5)  (31)[circle, minimum width =17pt, minimum height =17pt, inner sep=0.1pt, fill=blue!20,draw=black]{$b_2$};
	\node at(17.55,8.5)          {$\cdots$};
	\node at(17.55,7.5)  (32)[circle, minimum width =17pt, minimum height =17pt, inner sep=0.1pt, fill=blue!20,draw=black]{$b_n$};
	\node at(17.55,6)    (33)[circle, minimum width =17pt, minimum height =17pt, inner sep=0.1pt, fill=blue!20,draw=black]{$t_1$};
	\node at(17.55,5)   (34)[circle, minimum width =17pt, minimum height =17pt, inner sep=0.1pt, fill=blue!20,draw=black]{$t_2$};
	\node at(17.55,4)          {$\cdots$};
	\node at(17.55,3)   (35)[circle, minimum width =17pt, minimum height =17pt, inner sep=0.1pt, fill=blue!20,draw=black]{$t_n$};
	\node at(18.8,7)    (36)[circle, minimum width =17pt, minimum height =17pt, inner sep=0.1pt, fill=blue!20,draw=black]{$\times$};
	\node at(20,7)      (37){$\mathcal{G}_{\theta}(u)(z)$};
	\node at(12.5,4.5)  (38){$z$};
	
	\draw[->] (24) --(25);
	\draw[->] (25) --(26);
	\draw[->] (26) --(30);
	\draw[->] (26) --(31);
	\draw[->] (26) --(32);
	\draw[->] (27) --(33);
	\draw[->] (27) --(34);
	\draw[->] (27) --(35);
	\draw[->] (28) --(36);
	\draw[->] (29) --(36);
	\draw[->] (36) --(19.35,7);
	\draw[->] (38) --(27);
\end{tikzpicture}
\end{figure}
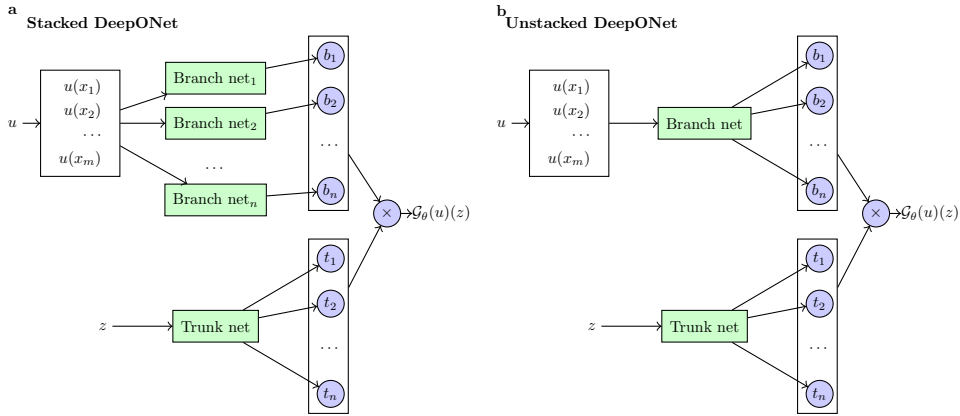

\subsection{Primal-dual method with DeepONets for (\ref{optimal_control})-(\ref{admissible_set})}
The DeepONets \cite{lu2021learning} provide a specialized deep learning framework to learn PDE solution operators.  For the convenience of readers, we give a brief overview of the
DeepONets, with a special focus on learning the solution operator of the state equation (\ref{state_eqn}).
 Typically, as shown in Figure \ref{fig:deepOnet}, the DeepONets architecture consists of two separate neural networks referred to as the ``branch" and ``trunk"
 networks, respectively. Then the DeepONets can be expressed as
$$
\mathcal{G}_{\theta}(u)(z)=\sum_{i=1}^nb_i(u)t_i(z)+b_0.
$$
Above, $\theta$ denotes the collection of all trainable weight and bias parameters in the branch and trunk networks. The vector $(b_1(u) , \cdots , b_n(u) )^\top$ is the output of the branch network with the input $\{u(x_j)\}_{j=1}^m$. For each input function $u$, $\{u(x_j)\}_{j=1}^m$ represent the evaluations at the fixed scattered sensors $\{x_j\}_{j=1}^m\subset\Omega\times(0,T)$. The vector $(t_1(z) , \cdots, t_n(z))^\top$ is the output of the trunk network with the continuous coordinates $z\in \Omega\times (0,T)$ as inputs; and $b_0 \in \mathbb{R}$ is a trainable bias.  Different from the fixed
locations $\{x_j\}_{j=1}^m\subset\Omega\times(0,T)$, the coordinates $z\in \Omega\times(0,T)$ may vary for different $u$. The
final output of the DeepONets is obtained by merging the outputs of the branch and trunk networks via an inner product.   The stacked DeepONet has one trunk network and $n$ stacked branch networks, while the unstacked DeepONet has one trunk network and one branch network.

Furthermore, we note that the DeepONets do not restrict the branch and trunk nets to any specific architecture. As $z\in \Omega\times(0,T)$ is usually low dimensional, a standard fully-connect neural network (FNN) is commonly used as the trunk net. The choice of the branch net depends on the structure of the input function $u$, and it can be chosen as a FNN, residual neural network, convolutional neural network (CNN), or a graph neural network (GNN).
For instance, if $\{u(x_j)\}_{j=1}^m$ is defined on a two-dimensional equispaced grid, then a CNN can be used; if $\{u(x_j)\}_{j=1}^m$  is given on an unstructured mesh, then a GNN can be used.
We refer to \cite{lu2021learning,lucomparison2022} for more discussions on the DeepONets.


Given a data set of different input functions $\{u_i\}$ and points $\{z_j\}$: $\{G(u_i)(z_j)\}_{1\leq i\leq N_1, 1\leq j\leq N_2}$ , we train the DeepONets by solving
\begin{equation}\label{loss_deeponet}
\min_{\theta}\mathcal{L}(\theta):=\frac{1}{N_1N_2}\sum_{i=1}^{N_1}\sum_{j=1}^{N_2}\|\mathcal{G}_{\theta}(u_i)(z_j)-G(u_i)(z_j)\|_{L^2(\Omega\times (0,T))}^2.
\end{equation}
Note that one
data point is a triplet $(u_i; z_j; G(u_i)(z_j))$, and thus one specific input function $u_i$ may appear in multiple data points with different $z_j$. For example, a dataset of size 400 can be generated from 20 $u$ trajectories, and each evaluates $G(u)(z)$ for 20 $z$ locations.

Using the DeepONets, we can obtain a surrogate model $y=\mathcal{S}_{\theta_s^*}(u)$ for the state equation $y=Su$, where $\theta_s^*$ is obtained by solving (\ref{loss_deeponet}) with $G$ replaced by the solution operator $S$. Similarly, we can also obtain a surrogate model $q=\mathcal{S}_{\theta_a^*}(p)$ for the adjoint equation $q=S^*p$. We thus specify (\ref{pdol}) as the following primal-dual method with DeepONets:
\begin{equation}\label{pdonet}
	\begin{aligned}
		u^{k+1}=P_{U_{ad}}\left(-\frac{\mathcal{S}_{\theta_a^*}({p}^k)-\frac{1}{r}u^k}{\alpha+\frac{1}{r}}\right),\quad
		p^{k+1}=(\mathcal{S}_{\theta_s^*}(2{u}^{k+1}-u^k)+\frac{1}{s}p^k-y_d)/(1+\frac{1}{s}).
	\end{aligned}
\end{equation}
Clearly, with the pre-trained surrogate models $y=\mathcal{S}_{\theta_s^*}(u)$ and $q=\mathcal{S}_{\theta_a^*}(p)$, one only needs to compute $\mathcal{S}_{\theta_a^*}({p}^k)$ and $\mathcal{S}_{\theta_s^*}(2{u}^{k+1}-u^k)$, and implement some simple algebraic operations.  Moreover,  given a different target $y_d$, the primal-dual method with DeepONets (\ref{pdonet}) can be directly applied to the resulting optimal control problem without solving any PDE.  Hence, the primal-dual method with DeepONets (\ref{pdonet}) is easy and cheap to implement. Finally, it is easy to see that the primal-dual method with DeepONets (\ref{pdonet}) for solving (\ref{optimal_control})-(\ref{admissible_set}) can be easily extended to other various optimal control problems modeled by (\ref{Basic_Problem}), see Example 3 in Section \ref{se: results_PDOL} for more discussions.

\subsection{Numerical results}\label{se: results_PDOL}
In this section, we discuss the implementation of the primal-dual method with DeepONets (\ref{pdonet}), and validate its effectiveness via some pedagogical numerical examples involving elliptic and parabolic control constrained optimal control problems.

\medskip
\noindent\textbf{Example 3.}
We consider the following elliptic control constrained optimal control problem:
\begin{equation}\label{model_1d_e}
	\underset{u\in L^2(\Omega), y\in L^2(\Omega)}{\min}~ J(y,u)=\frac{1}{2}\|y-y_d\|_{L^2(\Omega)}^2+\frac{\alpha}{2}\|u\|_{L^2(\Omega)}^2+\theta(u),
\end{equation}
where the state $y$ and the control $u$ satisfy the following state equation:
\begin{flalign}\label{state_1d_e}
		-\nu\Delta y+y=u+f~ \text{in}~ \Omega,\quad
		y=0~ \text{on}~ \Gamma.
\end{flalign}
Above, {$\Omega\subset \mathbb{R}^d(d\ge 1)$ is a convex polyhedral domain} with boundary $\Gamma:=\partial\Omega$, $y_d\in L^2(\Omega)$ is a given target, and $f\in H^{-1}(\Omega)$ is a given source term. The constant $\nu>0$ is the diffusion coefficient and  $\alpha>0$ is a regularization parameter.  We denote by $\theta(u):=I_{U_{ad}}(u)$ the indicator function of the admissible set $
	U_{ad}=\{u\in L^\infty(\Omega)| a\leq u(x)\leq b, ~\text{a.e.~in}~ \Omega \}\subset L^2(\Omega),$
where $a,b \in L^2(\Omega)$.

In our numerical experiments, we set $\Omega = (0,1)$, $\nu=1$, $\alpha=10^{-3}$, $a=-0.5$ and $b=0.5$. We further let
$$
\begin{aligned}
	&y=k_s\sin(\pi x), q=\alpha k_a \sin(2\pi x), u=\max\{a,\min\{b,-\frac{q}{\alpha}\}\}, \\
	&f=-u-\Delta y+y, y_d=y+\Delta q-q,
\end{aligned}
$$
where $k_s$ and $k_a$ are constants. Then, it is easy to show that $(u,y)^\top$ is the solution of problem (\ref{model_1d_e})-(\ref{state_1d_e}).  By choosing different $k_s$ and $k_a$, we can obtain different $y_d$ and thus specify a series of elliptic control constrained optimal control problems.

To obtain a surrogate model for (\ref{state_1d_e}), we  first consider constructing a neural operator  $\mathcal{N}_{\theta}$ by a DeepONet to approximate the solution operator $\bar{S}$ of the following elliptic equation
\begin{flalign}\label{state_1d_e0}
		-\nu\Delta y+y=u ~\text{in}~ \Omega, \quad
		y=0~\text{on}~ \Gamma.
\end{flalign}
Then, it is easy to see that $y=\mathcal{S}_{\theta_s^*}(u):=\mathcal{N}_{\theta^*}(u+f)$ is a surrogate model for (\ref{state_1d_e}). Moreover, since the state equation (\ref{state_1d_e}) is self-adjoint,  we can use $q=\mathcal{S}_{\theta^*_a}(p):=\mathcal{N}_{\theta^*}(p)$ as a surrogate model for the corresponding adjoint system of (\ref{model_1d_e})-(\ref{state_1d_e}).

We employ an unstacked DeepONet to construct $y=\mathcal{N}_{\theta^*}(u)$. Both the branch net and the trunk net
are fully-connected neural networks consisting of 2 hidden layers with 20 neurons per hidden layer and equipped with hyperbolic tangent activation functions.  We adapt the MATLAB codes used in \cite{li2020FNO} to generate a set of training data $\{u_i; z_j; \bar{S}(u_i)(z_j)\}_{1\leq i\leq N_1, 1\leq j\leq N_2}$.  For every $u_i$, $\{u_i(x_j)\}_{j=1}^m$ are the inputs of the branch network. We take $N_1=1000, N_2=m=65$, $\{x_j\}_{j=1}^m$ and $\{z_j\}_{j=1}^{N_2}$ are equi-spaced grids in [0,1]. We sample zero-boundary functions $\{u_i\}^{N_1}_{i=1}\in L^2(0,1)$ from a Gaussian random field with a Riesz kernel, i.e.,
$$
u_i \sim \mathcal{GR}(0, C),~\text{with}~ C=49^2(-\Delta+49 I)^{-2.5},
$$
where $\Delta$ and $I$ represent the Laplacian and the identity operator, respectively.
We then compute the solutions $\bar{S}(u_i)$ exactly in a Fourier space (see \cite{li2020FNO} for the details), and finally evaluate the values of $\bar{S}(u_i)(z_j)$ for every $(u_i, z_j)$.  Moreover, we modify the output $\mathcal{N}_{\theta}(u_i)(z_j)$ as $\mathcal{N}_{\theta}(u_i)(z_j)x(x-1)$ so that the homogeneous Dirichlet boundary condition $y(0)=y(1)=0$ is satisfied automatically. For training the neural networks, we implement 20000 iterations of the Adam \cite{kingma2015} with learning rate $\eta=10^{-3}$.  The training of the DeepONet is performed in Python utilizing the PyTorch framework. The training process is initialized using the default initializer of PyTorch. After the training process, we thus obtain the neural operator $\mathcal{N}_{\theta^*}$ and hence the surrogate models $y=\mathcal{N}_{\theta^*}(u+f)$ and $q=\mathcal{N}_{\theta^*}(p)$. We then obtain the following primal-dual method with DeepONet for solving (\ref{model_1d_e})-(\ref{state_1d_e}):
\begin{equation}\label{pdonet_elliptic}
	\begin{aligned}
		u^{k+1}=P_{U_{ad}}\left(-\frac{\mathcal{N}_{\theta^*}{p}^k)-\frac{1}{r}u^k}{\alpha+\frac{1}{r}}\right),\quad
		p^{k+1}=(\mathcal{N}_{\theta^*}(2{u}^{k+1}-u^k+f)+\frac{1}{s}p^k-y_d)/(1+\frac{1}{s}).
	\end{aligned}
\end{equation}

We implement (\ref{pdonet_elliptic}) to (\ref{model_1d_e})-(\ref{state_1d_e}) with different choices of $k_s$ and $k_a$. We set $r=2\times 10^3$ and $s=4\times 10^{-1}$ in (\ref{pdonet_elliptic}), and terminate the iteration if (\ref{stopping}) holds with $tol=10^{-5}$. The numerical results are reported in Table \ref{tab:iter_ex3} and  Figure \ref{fig:result_ol_ell}. First, it can be observed from Table \ref{tab:iter_ex3} that   (\ref{pdonet_elliptic}) converges fast and the iteration numbers are almost not affected by $k_s$ and $k_a$. Moreover,  the relative errors of $u$ and $y$ are very small for all cases under investigation, which, together with the results in Figure \ref{fig:result_ol_ell}, imply that the computed controls and the exact ones are in excellent agreement and they are visually indistinguishable.  From these results, we may conclude  that (\ref{pdonet_elliptic}) is efficient and robust enough to pursue highly accurate solutions for control constrained elliptic optimal control problems.

\begin{table}[htpb]
	\centering
	\caption{Numerical results for Example 3 ( Err(u)=$\frac{\|u^k-u\|_ {L^2(\Omega)} }{\|u\|_{L^2(\Omega)}}$,Err(y)=$\frac{\|y^k-y\|_{L^2(\Omega)}}{\|y\|_{L^2(\Omega)}}$ ).}\label{tab:iter_ex3}
	{\footnotesize\begin{tabular}{|c|c|c|c|c|c|c|}
			\hline
			&$k_s=-0.2$&$k_s=0.2$&$k_s=0.4$&$k_s=0.6$&$k_s=0.8$&$k_s=1$\\
			&$k_a=-1$&$k_a=1$&$k_a=2$&$k_a=3$&$k_a=4$&$k_a=5$\\
			\hline
			Iter&30 &29& 26& 25& 25& 26\\
			\hline
			$Err(u)$&$1.41\times 10^{-2}$& $6.68\times 10^{-3}$&$9.30\times 10^{-3}$ &$1.28\times 10^{-2}$ & $1.69\times 10^{-2}$& $7.64\times 10^{-3}$\\
			\hline
		    $Err(y)$&$1.46\times 10^{-3}$ &$1.76\times 10^{-3}$&$1.84\times 10^{-3}$& $1.86\times 10^{-3}$&$1.85\times 10^{-3}$ &$1.99\times 10^{-3}$ \\
			\hline
		\end{tabular}
	}
\end{table}

\begin{figure}[htpb]
	\caption{ Numerical and exact controls for Example 3 .}\label{fig:result_ol_ell}
	\centering
	\subfigure[$k_s=-0.2,k_a=-1$]{
		\includegraphics[width=0.3\textwidth]{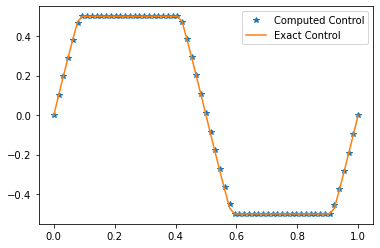}}
	\subfigure[$k_s=0.2,k_a=1$]{
		\includegraphics[width=0.3\textwidth]{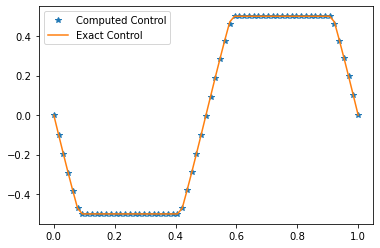}}
	\subfigure[$k_s=0.4,k_a=2$]{
		\includegraphics[width=0.3\textwidth]{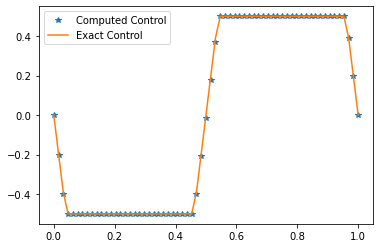}}
	\subfigure[$k_s=0.6,k_a=3$]{
		\includegraphics[width=0.3\textwidth]{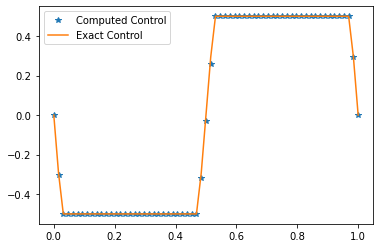}}
	\subfigure[$k_s=0.8,k_a=4$]{
		\includegraphics[width=0.3\textwidth]{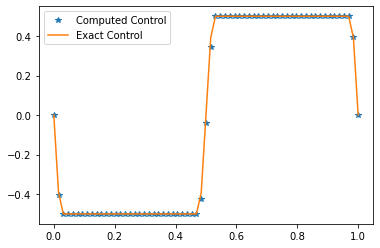}}
	\subfigure[$k_s=1,k_a=5$]{
		\includegraphics[width=0.3\textwidth]{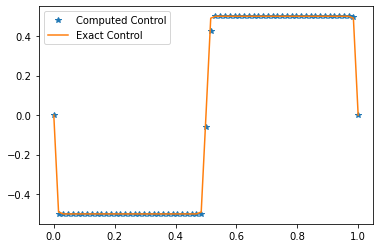}}
\end{figure}

\medskip
\noindent\textbf{Example 4.}  We consider the parabolic control constrained optimal control problem (\ref{ex1_Problem}). In particular, we set $\Omega=(0,1)$, $T=1$,  $\alpha=10^{-3}$ $a=-100$ and $b=100$. Let
$$
\begin{aligned}
	&y=k_s(e^t-1)\sin(\pi x),~ q=\alpha k_a(T-t)\sin(2\pi x), ~u=\max\{a,\min\{b,-\frac{q}{\alpha}\}\},\\
	&f=-u+\frac{\partial y}{\partial t}-\Delta y, ~y_d=y+\frac{\partial q}{\partial t}+\Delta q,
\end{aligned}
$$
where $k_s$ and $k_a$ are constants. Then, it is easy to show that $(u,y)^\top$ is the solution of problem (\ref{ex1_Problem}).  A series of parabolic control constrained optimal control problems can be specified by choosing different $k_s$ and $k_a$.

It is easy to see that implementing the primal-dual method with DeepONets (\ref{pdonet}) to  (\ref{ex1_Problem}) requires two surrogate models $y=\mathcal{S}_{\theta_s^*}(u)$ and $q=\mathcal{S}_{\theta_a^*}(p)$, respectively, for the state equation (\ref{ex1_state}) and the corresponding adjoint equation:
\begin{equation}\label{adjointP_ex}
	-\frac{\partial q}{\partial t}
	-\Delta q ={p}~ \text{in}
	~\Omega\times(0,T),~
	q=0~ \text{on}~ \Gamma\times(0,T),~
	q(T)=0.
\end{equation}
To this end, we first discretize (\ref{ex1_state}) and (\ref{adjointP_ex}) in time by the backward Euler method with the step size $\tau= T/N$, where $N$ is a positive integer. The resulting discretized state equation reads:
 $y_0=\phi$; for $n=1,\ldots,N$, with $y_{n-1}$ being known, we obtain $y_n$ from the
solution of the following linear elliptic problem:
\begin{flalign}\label{ell_step}
	\begin{aligned}
		-{\tau}\Delta{y}_n+{y}_n= \tau (f_n+u_n)+{y}_{n-1}~ \text{in}~ \Omega, \quad
		y_n=0~ \text{on}~ \Gamma,
	\end{aligned}
\end{flalign}
and the resulting discretized adjoint equation reads: $q(T)=0$; for $n=N-1,\ldots,0$, with $q_{n+1}$ being known, we obtain $q_n$ from the
solution of the following linear elliptic problem:
\begin{flalign}\label{ell2_step}
	\begin{aligned}
		-{\tau}\Delta{q}_n+{q}_n= \tau p_n+{q}_{n+1}~ \text{in}~ \Omega, \quad
		q_n=0~ \text{on}~ \Gamma,
	\end{aligned}
\end{flalign}
where we denote by $y_n$, $f_n$, $u_n$, $q_n$ and $p_n$ the approximate values of $y(n\tau)$, $f(n\tau)$, $u(n\tau)$, $q(n\tau)$ and $p(n\tau)$, respectively. It is easy to see that the elliptic equations (\ref{ell_step}) and (\ref{ell2_step}) have the same form as that of (\ref{state_1d_e0}). Hence, we can follow the same routine presented in Example 3 to construct two DeepONet surrogates for (\ref{ell_step}) and (\ref{ell2_step}). For the implementation of (\ref{pdonet}), we set $r=8\times 10^2$ and $s=4\times 10^{-1}$ and terminate the iteration if (\ref{stopping}) holds with $tol=10^{-5}$.

 The numerical results with respect to different $k_s$ and $k_a$ are reported in Table \ref{tab:iter_ex4} and Figure \ref{fig:result_ol_para}. Table \ref{tab:iter_ex4} shows that the primal-dual method with DeepONets (\ref{pdonet}) converges fast, with almost the same number of iterations for different values of $k_s$ and $k_a$. This suggests that the method is highly efficient and robust to the choices of $k_s$ and $k_a$. Additionally, the relative errors of $u$ and $y$ are very small across all test cases, which, in conjunction with the results presented in Figure \ref{fig:result_ol_ell}, indicate that the exact and computed controls are in excellent agreement and cannot be distinguished visually. Overall, these results demonstrate that the primal-dual method with DeepONets (\ref{pdonet}) is capable of producing highly accurate solutions.

\begin{table}[htpb]
	\centering
	\caption{Numerical results for Example 4.  (Err(u)=$\frac{\|u^k-u\|_{L^2(\Omega\times(0,T))}}{\|u\|_{L^2(\Omega\times(0,T))}}$,Err(y)=$\frac{\|y^k-y\|_{L^2(\Omega\times(0,T))}}{\|y\|_{L^2(\Omega\times(0,T))}}$ ).}\label{tab:iter_ex4}
	{\footnotesize\begin{tabular}{|c|c|c|c|c|c|}
			\hline
			&$k_s=0.3$&$k_s=0.4$&$k_s=0.6$&$k_s=-0.3$&$k_s=-0.5$\\
				&$k_a=500$&$k_a=600$&$k_a=700$&$ k_a=-500$&$k_a=-600$\\
			\hline
			Iter&38 &37& 37& 38& 40\\
			\hline
			Err(u)&$1.18\times 10^{-2}$& $1.11\times 10^{-2}$&$1.10\times 10^{-2}$ &$1.65\times 10^{-2}$ & $1.73\times 10^{-2}$\\
			\hline
			Err(y)&$9.67\times 10^{-2}$ &$7.18\times 10^{-2}$&$6.05\times 10^{-2}$& $1.04\times 10^{-1}$&$7.43\times 10^{-2}$  \\
			\hline
		\end{tabular}
	}
\end{table}

\begin{figure}[htpb]
	\caption{ Numerical and exact controls at $t=0.25$ (left), $0.5$ (middle), and $0.75$ (right)  for Example 4.}\label{fig:result_ol_para}
	\centering
	\subfigure[$k_s=0.3,k_a=500$]{
		\includegraphics[width=0.3\textwidth]{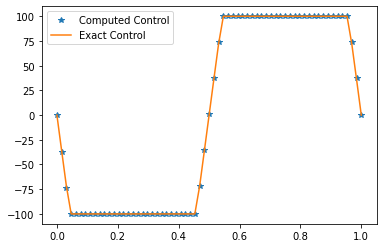}}
	\subfigure[$k_s=0.3,k_a=500$]{
		\includegraphics[width=0.3\textwidth]{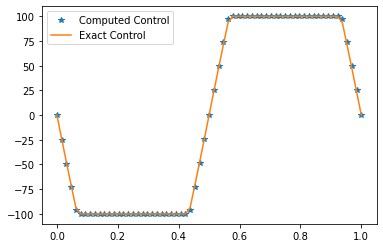}}
	\subfigure[$k_s=0.3,k_a=500$]{
		\includegraphics[width=0.3\textwidth]{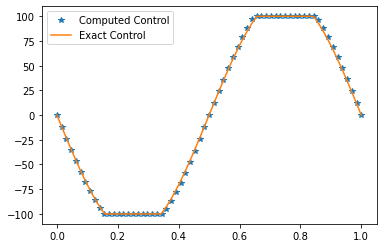}}
	\subfigure[$k_s=0.4,k_a=600$]{
		\includegraphics[width=0.3\textwidth]{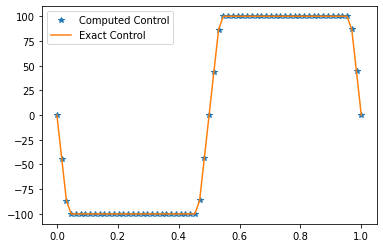}}
	\subfigure[$k_s=0.4,k_a=600$]{
		\includegraphics[width=0.3\textwidth]{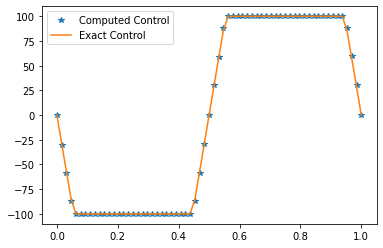}}
	\subfigure[$k_s=0.4,k_a=600$]{
		\includegraphics[width=0.32\textwidth]{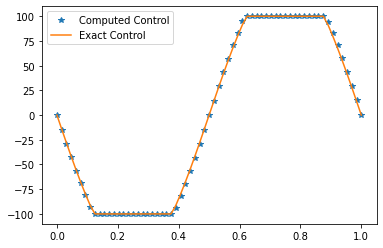}}
		\subfigure[$k_s=0.6,k_a=700$]{
		\includegraphics[width=0.3\textwidth]{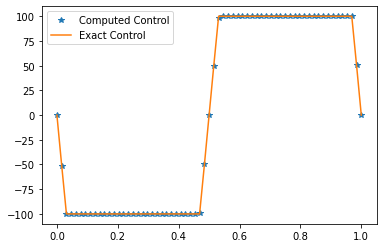}}
	\subfigure[$k_s=0.6,k_a=700$]{
		\includegraphics[width=0.3\textwidth]{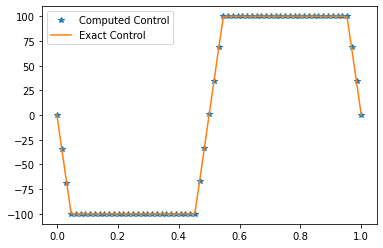}}
	\subfigure[$k_s=0.6,k_a=700$]{
		\includegraphics[width=0.3\textwidth]{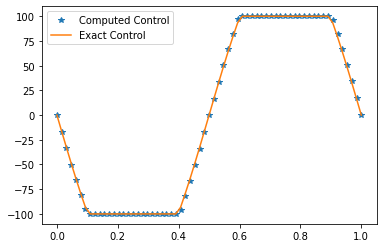}}
			\subfigure[$k_s=-0.3,k_a=-500$]{
		\includegraphics[width=0.3\textwidth]{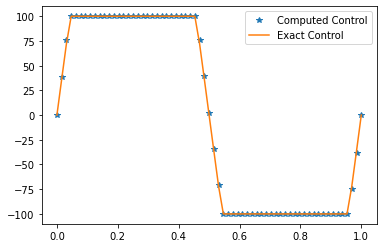}}
	\subfigure[$k_s=-0.3,k_a=-500$]{
		\includegraphics[width=0.3\textwidth]{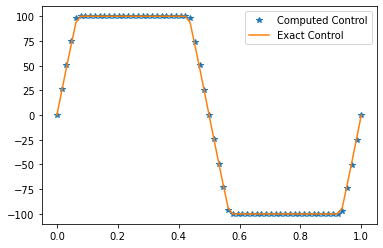}}
	\subfigure[$k_s=-0.3,k_a=-500$]{
		\includegraphics[width=0.3\textwidth]{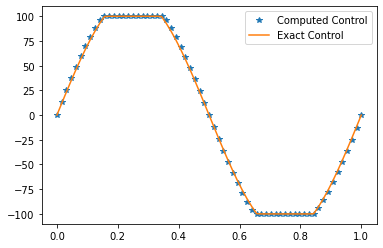}}
				\subfigure[$k_s=-0.5,k_a=-600$]{
		\includegraphics[width=0.3\textwidth]{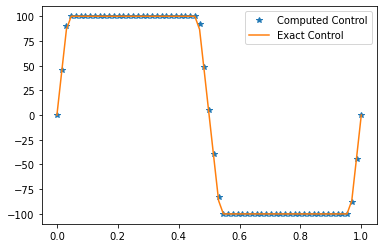}}
	\subfigure[$k_s=-0.5,k_a=-600$]{
		\includegraphics[width=0.3\textwidth]{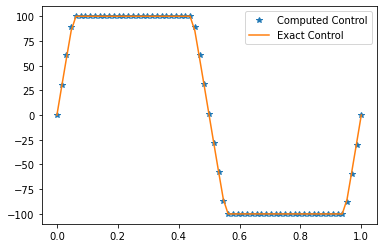}}
	\subfigure[$k_s=-0.5,k_a=-600$]{
		\includegraphics[width=0.3\textwidth]{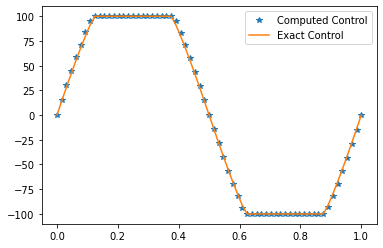}}
\end{figure}

\section{Conclusions and perspectives}\label{sec:conclusion}

We proposed two accelerated primal-dual methods for a general class of nonsmooth optimal control problems with partial differential equation (PDE) constraints. Both the accelerated methods keep the common advantages of primal-dual type methods. That is, different types of variables can be treated separately and thus the main computational load of each iteration is for two PDEs, while there is no need to solve high-dimensional and ill-conditioned saddle point systems or optimal control subproblems. For the accelerated primal-dual method with enlarged step sizes, it accelerates the primal-dual method in a simple and universal way, yet its convergence can be still proved rigorously. For the accelerated primal-dual method with operator learning, it is mesh-free and easy to implement. Indeed, surrogate models are constructed for the PDEs by deep neural networks, and once a neural operator is learned, only a forward pass of the neural networks is required to solve the PDEs. Efficiency of both the accelerated primal-dual methods is validated promisingly by numerical results.

Our work leaves interesting questions for future study. First, our philosophy of algorithmic design can be conceptually applied to optimal control problems with nonlinear PDE constraints. This could be achieved by combining the primal-dual method proposed in \cite{valkonen2014primal} with operator learning techniques. Second, our numerical results promisingly justifies the necessity to investigate some theoretical issues such as the convergence estimate and the error estimate for the approach of primal-dual methods with operator learning. Finally, we focused on the DeepONets to elaborate on our main ideas of accelerating primal-dual methods by operator learning techniques. It is interesting to consider other operator learning techniques in, e.g., \cite{cao2023,li2020FNO,li2020GNO,wang2021}.


\newpage

\bibliographystyle{siamplain}

\begin{thebibliography}{10}
	
	\bibitem{andrade2012multigrid}
	{ S.~G. Andrade and A.~Borz{\`\i}}, {\em Multigrid second-order accurate
		solution of parabolic control-constrained problems}, Computational
	Optimization and Applications, 51 (2012), pp.~835--866.
	
	
	\bibitem{attouch2008augmented}
{ H.~Attouch and M.~Soueycatt}, {\em Augmented Lagrangian and proximal alternating direction methods of multipliers in Hilbert spaces: applications to games, PDE's and control}, Pacific Journal of Optimization, 5 (2008),
pp.~17--37.

   \bibitem{barry2022}
J. Barry-Straume, A. Sarshar, A. A. Popov, and A. Sandu,  {\em Physics-informed neural networks for PDE-constrained optimization and control}, arXiv preprint arXiv:2205.03377, 2022.


\bibitem{bauschke2011}
H. H. Bauschke and P. L. Combettes, {\em Convex Analysis and Monotone Operator Theory in Hilbert Spaces}, Vol. 408, New York: Springer, 2011.

\bibitem{beck2019}
C. Beck, W. E, and A. Jentzen, {\em Machine learning approximation algorithms for high-dimensional fully nonlinear partial differential equations and second-order backward
stochastic differential equations}, Journal of Nonlinear Science, 29 (2019), pp. 1563--1619.
%


\bibitem{biccari2022}
U. Biccari, Y. Song, X. Yuan and E.  Zuazua, {\em A two-stage numerical approach for the sparse initial source identification of a diffusion-advection equation}. arXiv preprint arXiv:2202.01589, 2022.



\bibitem{cao2023}
Q. Cao, S. Goswami, and G.E. Karniadakis,  {\em LNO: Laplace neural operator for solving differential equations}, arXiv preprint, arXiv:2303.10528, 2023.
	
	\bibitem{chambolle2011first}
	{ A.~Chambolle and T.~Pock}, {\em A first-order primal-dual algorithm for
		convex problems with applications to imaging}, Journal of mathematical
	imaging and vision, 40 (2011), pp.~120--145.
	
	
	\bibitem{ciaramella2016}
	{ G. Ciaramella and A. Borz{\`\i}},  {\em A LONE code for the sparse control of quantum systems}, Computer Physics Communications, 200 (2016), pp.~312--323.
	

	
	\bibitem{clason2017primal}
	{ C.~Clason and T.~Valkonen}, {\em Primal-dual extragradient methods for
		nonlinear nonsmooth PDE-constrained optimization}, SIAM Journal on
	Optimization, 27 (2017), pp.~1314--1339.
	
	\bibitem{e2018}
	W. E and B. Yu, {\em The deep Ritz method: A deep learning-based numerical algorithm for solving
	variational problems}, Communications in Mathematics and Statistics, 6 (2018), pp. 1--12.
	
	\bibitem{elvetun2016}
	O. L. Elvetun, and B. F. Nielsen,  {\em The split Bregman algorithm applied to PDE-constrained optimization problems with total variation regularization}, Computational Optimization and Applications, 64 (2016), pp.~699--724.
%
	
	\bibitem{gabay1975dual}
	{ D.~Gabay and B.~Mercier}, {\em A dual algorithm for the solution of non linear variational problems via finite element approximation}, Computers \& Mathematics with Applications, 2 (1976): pp.~17--40.
	
	
	\bibitem{glowinski1994exact}
	{ R.~Glowinski and J.~Lions}, {\em Exact and approximate controllability for
		distributed parameter systems, Part I}, Acta Numerica, 3 (1994), pp.~269--378.
	
	\bibitem{glowinski1995exact}
	{ R.~Glowinski and J.~L.~Lions}, {\em Exact and approximate controllability for distributed parameter systems, Part II}, Acta Numerica, 4 (1995), pp.~159--328.
	
	

	

	
	\bibitem{glowinski2008exact}
	{ R.~Glowinski, J. L. Lions, and J.~He}, {\em Exact and Approximate
		Controllability for Distributed Parameter Systems: A Numerical Approach
		(Encyclopedia of Mathematics and its Applications)}, Cambridge University
	Press, 2008.
	
		\bibitem{glowinski1975approximation}
	{ R.~Glowinski and A.~Marroco}, {\em Sur l'approximation, par {\'e}l{\'e}ments finis d'ordre un, et la r{\'e}solution, par p{\'e}nalisation-dualit{\'e} d'une classe de probl{\`e}mes de dirichlet non lin{\'e}aires}, Revue fran{\c{c}}aise d'automatique, informatique, recherche
	op{\'e}rationnelle. Analyse Num{\'e}rique, 9 (1975), pp.~41--76.
	
	\bibitem{GSY2019}{ R.~Glowinski, Y. Song and X. Yuan}, {\em An ADMM numerical approach to linear parabolic state constrained optimal control problems}, Numerische Mathematik, 144 (2020), pp.~931--966.
	
	\bibitem{glowinski2022}
	R.~Glowinski, Y. Song, X. Yuan, and H. Yue, {\em Application of the alternating direction method of multipliers to control constrained parabolic optimal control problems and beyond.} Ann. Appl. Math, 38 (2022), pp. 115-158.
	
	\bibitem{goldstein2015adaptive}
	{ T. Goldstein, M. Li and X. Yuan},
	{\em  Adaptive primal-dual splitting methods for statistical learning and
	image processing}.
	In  Advances in Neural Information Processing Systems, (2015),
	pp.~2089--2097.
		
		\bibitem{han2018}
		J. Han, A. Jentzen, and W. E, {\em Solving high-dimensional partial differential equations using
		deep learning}, Proceedings of the National Academy of Sciences, 115 (2018), pp. 8505--8510.
		
	
	\bibitem{haoBilevel2022}
	Z. Hao, C. Ying, H. Su,  J. Zhu, J. Song, and Z. Cheng, {\em Bi-level physics-informed neural networks for PDE constrained optimization using Broyden's hypergradients}, arXiv preprint arXiv:2209.07075, 2022.
	

	
	\bibitem{he2022}
	B. He, F. Ma, S. Xu and X. Yuan, {\em A generalized primal-dual algorithm with improved convergence condition for saddle point problems}, SIAM Journal on Imaging Sciences, 15 (2022), pp. 1157--1183.
	
	
	
%
	
		
	
	\bibitem{hintermuller2002primal}
	{ M.~Hinterm{\"u}ller, K.~Ito, and K.~Kunisch}, {\em The primal-dual active
		set strategy as a semismooth newton method}, SIAM Journal on Optimization, 13
	(2002), pp.~865--888.
	
	\bibitem{kunisch2004}
	K. Kunisch and M. Hintermüller, {\em Total bounded variation regularization as a bilaterally constrained optimization problem}, SIAM Journal on Applied Mathematics, 64 (2004), pp.~1311--1333.
	
		
	\bibitem{hinze2008optimization}
	{ M.~Hinze, R.~Pinnau, M.~Ulbrich, and S.~Ulbrich}, {\em Optimization with
		PDE Constraints}, Vol.~23, Springer Science \& Business Media, 2008.
	
	\bibitem{hwang2021solving}
	R.  Hwang, J.~Y. Lee, J.~Y. Shin,  and H.~J. Hwang,
{\em Solving PDE-constrained control problems using operator learning}, arXiv preprint arXiv:2111.04941, 2021.
	
	
	\bibitem{kingma2015}
	D. P. Kingma and J. Ba, {\em Adam: A method for stochastic optimization}, in the 3rd International
	Conference on Learning Representations, 2015; preprint available from https://arxiv.org/
	abs/1412.6980
	
	\bibitem{kroner2011}
	A. Kr{\"o}ner, K. Kunisch and B. Vexler, {\em Semismooth Newton methods for optimal
	control of the wave equation with control constraints,} SIAM Journal on Control and Optimization, 49
	(2011), pp. 830-858.
	
	\bibitem{kocachki2021}
	N. Kovachki, Z. Li, B. Liu, K. Azizzadenesheli,
	K. Bhattacharya, A. Stuart, and A. Anandkumar, {\em Neural operator: Learning
	maps between function spaces}. arXiv preprint arXiv:2108.08481, 2021.
	
	\bibitem{KR2002}
{  K. Kunisch and  A. R{\"o}sch},  \emph{Primal-dual active set strategy for a general class of constrained optimal control problems}, SIAM Journal on Optimization, 13 (2002), pp.~321--334.
	
	\bibitem{li2020FNO}
	Z. Li, N. Kovachki, K. Azizzadenesheli, B. Liu, K. Bhattacharya, A. Stuart, and A. Anandkumar, {\em Fourier neural operator for parametric partial differential equations.} arXiv preprint arXiv:2010.08895, 2020.
	
	\bibitem{li2020GNO}
	Z. Li, N. Kovachki, K. Azizzadenesheli, B. Liu, K. Bhattacharya, A. Stuart, and A. Anandkumar,  {\em Neural operator: Graph kernel network for partial differential equations.} arXiv preprint arXiv:2003.03485, 2020.
	
	\bibitem{lions1971optimal}
	{ J.~L. Lions}, {\em Optimal Control of Systems Governed by Partial
		Differential Equations (Grundlehren der Mathematischen Wissenschaften)},
	vol.~170, Springer Berlin, 1971.
	
	\bibitem{liu2023jcp}
	S. Liu, S. Osher, W. Li, and C. W.  Shu,  {\em A primal-dual approach for solving conservation laws with implicit in time approximations}, Journal of Computational Physics, 472 (2023), pp.~111654.
	
   \bibitem{liu2023arxiv}
   S. Liu, S.Liu, S. Osher, and W. Li,  {\em A first-order computational algorithm for reaction-diffusion type equations via primal-dual hybrid gradient method,} arXiv preprint, arXiv:2305.03945, 2023.
	
	\bibitem{lu2021learning}
	L. Lu, P. Jin, G. Pang, Z. Zhang, and G. E. Karniadakis, {\em Learning nonlinear operators via DeepONet based on the universal approximation theorem of operators}, Nature Machine Intelligence, 3(2021), pp.~218--229.
	
	\bibitem{ludeepxde2021}
	L. Lu, X. Meng, Z. Mao, and G.E. Karniadakis, {\em DeepXDE: A deep learning library for solving differential equations}, SIAM Review, 63 (2021), pp.~208--228.
	
	\bibitem{lucomparison2022}
	L. Lu, X. Meng, S. Cai, Z. Mao, S. Goswami, Z. Zhang, and G. E. Karniadakis. {\em A comprehensive and fair comparison of two neural operators (with practical extensions) based on fair data.} Computer Methods in Applied Mechanics and Engineering, 393 (2022), pp.~114778.
	
	\bibitem{lye2021iterative}
	{ K.~O. Lye, S. Mishra, D. Ray, and P. Chandrashekar,} {\em Iterative surrogate model optimization (ISMO): An active learning algorithm for PDE constrained optimization with deep neural networks}, Computer Methods in Applied Mechanics and Engineering 374 (2021), pp.~113575.
	
	\bibitem{mowlayi2021}
	S. Mowlavi and S. Nabi. {\em Optimal control of PDEs using physics-informed neural networks}, Journal of Computational Physics, 473  (2023), pp.~111731.
	
	
	\bibitem{pang2019}
	G. Pang, L. Lu, and G. E. Karniadakis, {\em fPINNs: Fractional physics-informed neural networks}, SIAM Journal on Scientific Computing, 41 (2019), pp. A2603--A2626.
	
	\bibitem{pearson2017}
J. W. Pearson and J. Gondzio, {\em Fast interior point solution of quadratic programming
problems arising from PDE-constrained optimization}, Numerische Mathematik, 137
(2017), pp. 959--999
	
	
		\bibitem{pearson2012regularization}
	{ J.~W.~Pearson, M.~Stoll, and A.~J. Wathen}, {\em Regularization-robust preconditioners for time-dependent PDE-constrained optimization problems},
	SIAM Journal on Matrix Analysis and Applications, 33 (2012), pp.~1126--1152.
	

	
	\bibitem{porcelli2015preconditioning} { M.~Porcelli, V.~Simoncini, and M.~Tani}, {\em Preconditioning of active-set Newton methods for PDE-constrained optimal control problems}, SIAM Journal on Scientific Computing, 37 (2015), pp.~S472--S502.
%
\bibitem{pougkakiotis2020}
S. Pougkakiotis,  J. W. Pearson, S. Leveque,  and J. Gondzio,   {\em Fast solution methods for convex quadratic optimization of fractional differential equations}, SIAM Journal on Matrix Analysis and Applications, 41(2020), pp.~1443--1476.

\bibitem{raissi2019physics}
{ M. Raissi, P. Perdikaris, and G.~E. Karniadakis,}
{\em Physics-informed neural networks: A deep learning framework for
	solving forward and inverse problems involving nonlinear partial differential
	equations}, Journal of Computational physics, 378 (2019), pp. 686--707.


	
	
	\bibitem{schiela2014operator}
	{ A.~Schiela and S.~Ulbrich}, {\em Operator preconditioning for a class of
		inequality constrained optimal control problems}, SIAM Journal on
	Optimization, 24 (2014), pp.~435--466.
	
	\bibitem{sirignano2018dgm}
	J. Sirignano and K. Spiliopoulos,  {\em DGM: A deep learning algorithm for solving partial differential equations}. Journal of Computational Physics, 375 (2018), pp. 1339--1364.
	
%
	\bibitem{song2023admmpinns}
	Y. Song, Y. Yuan, and H. Yue, {\em The ADMM-PINNs algorithmic framework for nonsmooth PDE-constrained optimization: a deep learning approach}, arXiv preprint arXiv:2302.08309,2023.

	
	\bibitem{stadler2009elliptic}
	{G. Stadler},
	{\em Elliptic optimal control problems with ${L}^1$-control cost and
	applications for the placement of control devices},
	Computational Optimization and Applications, 44 (2009), pp.~159--181.
	
	\bibitem{stoll2013one}
	{ M.~Stoll}, {\em One-shot solution of a time-dependent time-periodic
		PDE-constrained optimization problem}, IMA Journal of Numerical Analysis, 34
	(2013), pp.~1554--1577.
	
	\bibitem{sun2022}
	Y. Sun, U. Sengupta, and M. Juniper. Physics-informed Deep Learning for simultaneous Surrogate Modelling and PDE-constrained Optimization. Bulletin of the American Physical Society, 2022.
%
	
	
	\bibitem{tian2018convergence}
	{ W.~Tian and X.~Yuan}, {\em Convergence analysis of primal--dual based
		methods for total variation minimization with finite element approximation},
	Journal of Scientific Computing, 76 (2018), pp.~243--274.
	
	
	\bibitem{troltzsch2010optimal}
	{ F.~Tr{\"o}ltzsch}, {\em Optimal Control of Partial Differential Equations:
		Theory, Methods, and Applications}, Vol.~112, AMS,
	2010.
	
		\bibitem{ulbrich2011semismooth}
	{ M.~Ulbrich}, {\em Semismooth Newton Methods for Variational Inequalities and Constrained Optimization Problems in Function Spaces}, Vol.~11, SIAM, 2011.
	
	
	\bibitem{valkonen2014primal}
	{ T.~Valkonen}, {\em A primal--dual hybrid gradient method for nonlinear
		operators with applications to MRI}, Inverse Problems, 30 (2014), pp.~055012.
	
	\bibitem{wachsmuth2011}
	{ G. Wachsmuth and D. Wachsmuth}, {\em Convergence and regularization results for optimal control problems with sparsity functional}, ESAIM: Control, Optimisation and Calculus of Variations, 17 (2011), pp.~858--886.
	
	\bibitem{wang2021fast}
	{S. Wang, M.~A. Bhouri, and P. Perdikaris,} {\em Fast PDE-constrained optimization via self-supervised operator earning}, arXiv preprint arXiv:2110.13297, 2021.
	
	\bibitem{wang2021}
	S. Wang, H. Wang, and P. Perdikaris, {\em Learning the solution operator of parametric partial differential equations with physics-informed DeepONets}, Science advances, 7(2021), pp.~eabi8605.


\bibitem{xue2020}
T. Xue, A. Beatson, S. Adriaenssens, and R. Adams,  {\em Amortized finite element analysis for fast PDE-constrained optimization}, In International Conference on Machine Learning, PMLR, 2020, pp. 10638--10647.
	
	\bibitem{yu2022gPINN}
	J. Yu, L. Lu, X. Meng, and G. E. Karniadakis, {\em Gradient-enhanced physics-informed neural networks for forward and inverse PDE problems}, Computer Methods in Applied Mechanics and Engineering, 393 (2022), pp.~114823.
	
	
	\bibitem{zhang2017}
	K. Zhang, J. Li, Y. Song,  and X. Wang, {\em An alternating direction method of multipliers
	for elliptic equation constrained optimization problem}, Science China Mathematics, 60 (2017),
	pp. 361--378.

	

	

	



	
	



	

	

	

	


	
	

	

	
	

	
	

	

	
\end{thebibliography}

\end{document}